\documentclass[a4paper,12pt,reqno]{amsart}
\usepackage{lmodern,amsmath,amssymb,mathtools,amsthm}
\usepackage{enumerate}
\usepackage{graphics,graphicx}
\usepackage{bigints}
\usepackage{geometry}
\usepackage{wasysym}
\usepackage{cleveref}
\usepackage{floatrow,morefloats}
\usepackage{epsfig}
\usepackage{epstopdf}
\usepackage{caption,subcaption}
\newtheorem{lemma}{Lemma}
\newtheorem{theorem}{Theorem}
\newtheorem{corollary}{Corollary}
\newtheorem{remark}{Remark}

\numberwithin{equation}{section}
\numberwithin{definition}{section}
\numberwithin{lemma}{section}
\numberwithin{theorem}{section}
\numberwithin{corollary}{section}
\numberwithin{example}{section}

\usepackage{blindtext}

\title{Differential Subordinations for Starlike Functions Associated With A Nephroid Domain}

\author{Lateef Ahmad Wani$^\dagger$}
   \address{$^\dagger$Department of Mathematics\\ Indian Institute of Technology, Roorkee-247667, Uttarakhand, India}
   \email{lateef17304@gmail.com}

\author{A. Swaminathan$^\ddagger$}
   \address{$^\ddagger$Department of Mathematics\\ Indian Institute of Technology, Roorkee-247667, Uttarakhand, India}
   \email{mathswami@gmail.com, a.swaminathan@ma.iitr.ac.in}
\allowdisplaybreaks
\bigskip

\begin{document}

\begin{abstract}
Let $\mathcal{A}$ be the set of all analytic functions $f$ defined in the open unit disk $\mathbb{D}$ and satisfying $f(0)=f'(0)-1=0$. In this paper, we consider the function $\varphi_{\scriptscriptstyle {Ne}}(z):=1+z-z^3/3$, which maps the unit circle $\{z:|z|=1\}$ onto a $2$-cusped curve called nephroid given by
$\left((u-1)^2+v^2-\frac{4}{9}\right)^3-\frac{4 v^2}{3}=0$, and the function class $\mathcal{S}^*_{Ne}$ defined as
\begin{align*}
\mathcal{S}^*_{Ne}:=\left\{f\in\mathcal{A}:\frac{zf'(z)}{f(z)}\prec\varphi_{\scriptscriptstyle {Ne}}(z)\right\},
\end{align*}
where $\prec$ denotes subordination. We obtain sharp estimates on $\beta\in\mathbb{R}$ so that the first-order differential subordination
\begin{align*}
1+\beta\frac{zp'(z)}{p^j(z)}\prec\mathcal{P}(z), \quad j=0,1,2
\end{align*}
implies $p\prec\varphi_{\scriptscriptstyle{Ne}}$, where $\mathcal{P}(z)$ is certain Carath\'{e}odory function with nice geometrical properties and $p(z)$ is analytic satisfying $p(0)=1$. Moreover, we use properties of Gaussian hypergeometric function in order to get the subordination $p\prec\varphi_{\scriptscriptstyle{Ne}}$ whenever $p(z)+\beta zp'(z)\prec\sqrt{1+z}$ or $1+z$. As applications, we establish sufficient conditions for $f\in\mathcal{A}$ to be in the class $\mathcal{S}^*_{Ne}$.  	
\end{abstract}


\maketitle

\markboth{Lateef Ahmad Wani and A. Swaminathan}{Differential subordinations for starlike functions associated with a nephroid domain}


\section{Introduction}
Let $\mathcal{H}$ be the collection of all analytic functions defined on the open unit disk $\mathbb{D}:=\left\{z\in\mathbb{C}:|z|<1\right\}$, where $\mathbb{C}$ denotes the complex plane. A function $f\in\mathcal{H}$ satisfying $f(0)=1$ and $\mathrm{Re}\left(f(z)\right)>0$ for every $z\in\mathbb{D}$ is called a Carath\'{e}odory function. Let
 $\mathcal{A}$ be the totality of analytic functions $f\in\mathcal{H}$ satisfying the normalization conditions $f(0)=0$ and $f'(0)=1$. Obviously, each function $f\in\mathcal{A}$ is of the form $f(z)=z+\sum_{n=2}^{\infty}a_nz^n,\,a_n\in\mathbb{C}$.
 Let $\mathcal{S}\subset\mathcal{A}$ denote the family of one--one (univalent) functions defined on $\mathbb{D}$. Further, let $\mathcal{S}^*$ and $\mathcal{C}$ be, respectively, the well-known classes of starlike and convex functions. The functions in $\mathcal{S}^*$ (or $\mathcal{C}$) are analytically characterized by the condition that for each $z\in\mathbb{D}$, the quantity $zf'(z)/f(z)$ (or $1+zf''(z)/f'(z)$) lies in the interior of the half-plane $\mathrm{Re}(w)>0$. Let $f,g\in\mathcal{H}$, by $f\prec g$ we mean $f$ is {\it subordinate} to $g$, which implies $f(z)=g(w(z))$ whenever there exists a function $w\in\mathcal{H}$ satisfying $w(0)=0$ and $|w(z)|<1$ for $z\in\mathbb{D}$. If $f\prec g$, then $f(0)=g(0)$ and $f(\mathbb{D})\subset g(\mathbb{D})$. Moreover, if the function $g(z)$ is univalent, then the concept $f\prec g$ and the property $f(0)=g(0)$ with $f(\mathbb{D})\subset g(\mathbb{D})$ are equivalent. Let $\Psi:\mathbb{C}^2\times\mathbb{D}\to\mathbb{C}$ be a complex analytic function, and let $u\in\mathcal{H}$ be univalent. A function $p\in\mathcal{H}$ is said to satisfy the first-order differential subordination if
\begin{align}\label{Def-Diff-Subord-Psi}
\Psi(p(z),\,zp'(z);\,z)\prec u(z), \qquad z\in\mathbb{D}.
\end{align}
 If $q:\mathbb{D}\to\mathbb{C}$ is univalent and $p{\prec}q$ for all $p$ satisfying \eqref{Def-Diff-Subord-Psi}, then $q$ is said to be a dominant of the differential subordination \eqref{Def-Diff-Subord-Psi}. A dominant $\tilde{q}$ that satisfies $\tilde{q}\prec{q}$ for all dominants $q$ of \eqref{Def-Diff-Subord-Psi} is called the best dominant of \eqref{Def-Diff-Subord-Psi}. If $\tilde{q}_1$ and $\tilde{q}_2$ are two best dominants of \eqref{Def-Diff-Subord-Psi}, then $\tilde{q}_2(z)=\tilde{q}_1(e^{i\theta}z)$ for some $\theta\in\mathbb{R}$ i.e., the best dominant is unique up to the rotations of $\mathbb{D}$. For further details related to differential subordinations, we refer to the monograph of Miller and Mocanu \cite{Miller-Mocanu-Book-2000-Diff-Sub} (see also \cite{Bulboaca-2005-Diff-Sub-Book}).

For $p\in\mathcal{H}$ satisfying $p(0)=1$, Nunokawa et al. \cite{Nunokawa-Owa-1989-One-criterion-PAMS} verified that the subordination $1+zp'(z)\prec1+z$ implies $p(z)\prec1+z$. As a consequence, they \cite{Nunokawa-Owa-1989-One-criterion-PAMS} gave a criterion for a normalized analytic function to be univalent in $\mathbb{D}$. In 2007, Ali et al. \cite{Ali-Ravi-2007-Janowski-Starlikeness-IJMMS} replaced $1+z$ by $(1+Dz)/(1+Ez)$ and obtained the conditions (non-sharp) on the parameter $\beta\in\mathbb{R}$ in terms of $A,B,D,E\in[-1,1]\;(B<A \text{ and } E<D)$ so that the following subordination implication holds:
$$1+\beta\frac{zp'(z)}{p^j(z)}\prec\frac{1+Dz}{1+Ez} \implies p(z)\prec\frac{1+Az}{1+Bz}, \quad j=0,1,2.$$
 As a result, certain sufficient conditions for Janowski starlikeness were established. In 2012, Ali et al. \cite{Ali-Ravi-2012-Diff-Sub-LoB-TaiwanJM} determined the conditions on the real $\beta$ so that the subordination $p(z)\prec\sqrt{1+z}$ holds true whenever the subordination $1+\beta{zp'(z)/p^j(z)}\prec\sqrt{1+z}\;(j=0,1,2)$ holds. In 2013, Kumar et al. \cite{Kumar-Ravi-2013-Suff-Conditions-LoB-JIA} gave non-sharp bounds for $\beta\in\mathbb{R}$ such that $1+\beta{zp'(z)/p^j(z)}\prec(1+Dz)/(1+Ez)$ implies $p(z)\prec\sqrt{1+z}$, where $D,E\in\mathbb{R}$ with $-1<E<D\leq1$. Later, Omar and Halim \cite{Omar-Halim-2013-Diff-Sub-Sokol-Stankiewicz-Class} studied this problem of Kumar et al. \cite{Kumar-Ravi-2013-Suff-Conditions-LoB-JIA} for $D\in\mathbb{C}$ with $|D|\leq1$ and $-1<E<1$.
The subordination results proved in \cite{Ali-Ravi-2012-Diff-Sub-LoB-TaiwanJM,Omar-Halim-2013-Diff-Sub-Sokol-Stankiewicz-Class, Kumar-Ravi-2013-Suff-Conditions-LoB-JIA} provide sufficient conditions for $f\in\mathcal{A}$ to be in the starlike class $\mathcal{S}^*_L$ of functions associated with the leminiscate of Bernoulli introduced by Sok\'{o}{\l} and Stankiewicz \cite{Sokol-J.Stankwz-1996-Lem-of-Ber}. In 2018, Kumar and Ravichandran \cite{SushilKumar-Ravi-2018-Sub-Positive-RP-CAOT} determined {\it sharp} estimates on the real $\beta$ in order that the subordination
$1+\beta{zp'(z)/p^j(z)}\prec\mathcal{P}(z)$ ensures $p(z)\prec{e^z},\,(1+Az)/(1+Bz)$ for a handful of Carath\'{e}odory functions $\mathcal{P}(z)$ with interesting geometries.
%
%
The results proved in \cite{SushilKumar-Ravi-2018-Sub-Positive-RP-CAOT} yield, in particular, certain conditions that are sufficient for $f\in\mathcal{A}$ to belong to the function class $\mathcal{S}^*_e$ related to the exponential function $e^z$ introduced by Mendiratta et al. \cite{Mendiratta-Ravi-2015-Expo-BMMS}.
Recently, Ahuja et al. \cite{Ahuja-Ravi-2018-App-Diff-Sub-Stud-Babe-Bolyai} obtained sharp bounds on $\beta$ so that the differential subordination $1+\beta{zp'(z)/p^j(z)}\prec\sqrt{1+z}\;(j=0,1,2)$ implies $p(z)\prec: \sqrt{1+z},\, 1+\sin{z},\, 1+4z/3+2z^2/3,\, z+\sqrt{1+z^2},\, (1+Az)/(1+Bz)$, where $-1<B<A<1$. Similar problems of subordination implications have been studied in
 \cite{Bohra-Ravi-2019-Diff-Subord-Hacet.J,
 	Cho-Ravi-2018-Diff-Sub-Booth-Lem-TJM,
 	Sokol-2007-Subord-AMC}.\\

Motivated by the aforesaid literature, in this paper, we consider the Carath\'{e}odory function $\varphi_{\scriptscriptstyle{Ne}}:\mathbb{D}\to\mathbb{C}$ defined as
$\varphi_{\scriptscriptstyle{Ne}}(z):=1+z-z^3/3$,
and the Ma-Minda type (see \cite{Ma-Minda-1992-A-unified-treatment}) function class $\mathcal{S}^*_{Ne}$ associated with it given by
\begin{align*}
\mathcal{S}^*_{Ne}:=\left\{f\in\mathcal{A}:\frac{zf'(z)}{f(z)}\prec\varphi_{\scriptscriptstyle {Ne}}(z)\right\}.
\end{align*}
Our problem is to determine sharp estimates on $\beta\in\mathbb{R}$ so that the first-order differential subordination
\begin{align}\label{Subordination-Problem-Main-SRN}
1+\beta\frac{zp'(z)}{p^j(z)}\prec\mathcal{P}(z), \quad z\in\mathbb{D}, \quad j\in\{0,1,2\}
\end{align}
implies $p(z)\prec\varphi_{\scriptscriptstyle{Ne}}(z)$, where $\mathcal{P}:\mathbb{D}\to\mathbb{C}$ is some analytic function with positive real part and has certain nice geometric properties. Furthermore, the starlike properties of the classical hypergeometric function ${_2F_1}$ are used to find the sharp bound on $\beta\in\mathbb{R}$ such that the differential subordination
\begin{align*}
p(z)+\beta{zp'(z)}\prec\sqrt{1+z},\, \text{ or }\, 1+z, \qquad z\in\mathbb{D},
\end{align*}
implies $p(z)\prec\varphi_{\scriptscriptstyle{Ne}}(z)$. All these results in turn yield conditions that sufficiently ensure that the function $f\in\mathcal{A}$ is a member of the function class $\mathcal{S}^*_{Ne}$.

The function $\varphi_{\scriptscriptstyle{Ne}}(z)$ and the associated class $\mathcal{S}^*_{Ne}$ were recently introduced in \cite{Wani-Swami-Nephroid-Basic,Wani-Swami-Radius-Problems-Nephroid-RACSAM}. It was proved \cite{Wani-Swami-Nephroid-Basic} that the function $\varphi_{\scriptscriptstyle{Ne}}(z)$ maps the boundary  $\partial\mathbb{D}$ of the unit disk $\mathbb{D}$ univalently onto the nephroid, a $2$--cusped kidney--shaped curve, given by
\begin{align}\label{Equation-of-Nephroid}
\left((u-1)^2+v^2-\frac{4}{9}\right)^3-\frac{4 v^2}{3}=0.
\end{align}
Indeed, for $-\pi<t\leq\pi$, we have
\begin{align*}
u+iv=\varphi_{\scriptscriptstyle {Ne}}(e^{it})=1+\cos{t}-(\cos{3t})/3+i\left(\sin{t}-(\sin{3t})/3\right),
\end{align*}
which on separating real and imaginary parts gives
\begin{align*}
(u-1)^2+v^2=\frac{10}{9}-\frac{2}{3}\cos{2t}
               =\frac{4}{9}+\left(\frac{4}{3}\left(\frac{4}{3}\sin^3{t}\right)^2\right)^\frac{1}{3}
                 =\frac{4}{9}+\left(\frac{4}{3}v^2\right)^\frac{1}{3}
\end{align*}
and yields the equation \eqref{Equation-of-Nephroid}. Geometrically, a nephroid is the locus of a point fixed on the circumference of a circle of radius $\rho$ that rolls (without slipping) on the outside of a fixed circle having radius $2\rho$.
First studied by Huygens and Tschirnhausen in 1697, the nephroid curve was shown to be the catacaustic (envelope of rays emanating from a specified point) of a circle when the light source is at infinity. In 1692, Jakob Bernoulli had shown that the nephroid is the catacaustic of a cardioid for a luminous cusp. However, the word nephroid was first used by Richard A. Proctor in 1878 in his book ``The Geometry of Cycloids". For further details related to the nephroid curve, we refer to \cite{Lockwood-Book-of-Curves-2007, Yates-1947-Handbook-Curves}.
\vspace{-1em}
\begin{figure}[H]
	\includegraphics[scale=0.7]{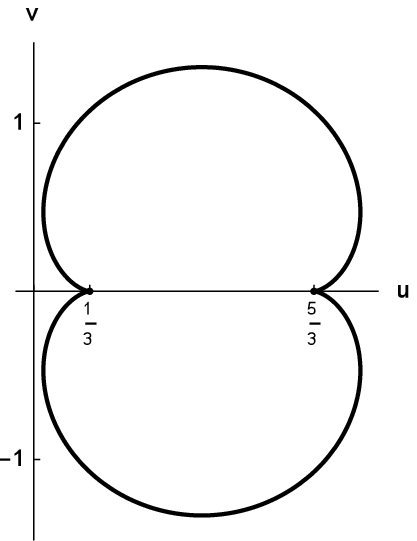}
	\caption{Nephroid: The Boundary curve of
		$\varphi_{\scriptscriptstyle{Ne}}(\mathbb{D})$.}
	\label{Figure-Nephroid}	
\end{figure}
The structure of rest of the paper is as follows. In \Cref{Section-Lemniscate-Regions-SRN}, the Carath\'{e}odory function $\mathcal{P}(z)$ in \eqref{Subordination-Problem-Main-SRN} is taken to be a certain function which maps $\mathbb{D}$ onto a convex domain, while in \Cref{Section-Non-Convex-Cusped-Regions-SRN}, the region $\mathcal{P}(\mathbb{D})$ is considered as non-convex, either cusped or dimpled.  
%
%
In \Cref{Section-Janowski-Functions-SRN},
the implications of \eqref{Subordination-Problem-Main-SRN} are provided for $\mathcal{P}(z)$ being the Janowski class of functions of the form $(1+Az)/(1+Bz),\,-1\leq B<A\leq1$. The properties of Gaussian hypergeometric functions are used in \Cref{Section-GHGFs-SRN} to discuss the outcome of the differential subordinations of the form $p(z)+\beta zp'(z)\prec\sqrt{1+z}$, or $1+z$. 

To prove the results, an extensive use of the following lemma is made.
\begin{lemma}[{\cite[Theorem 3.4h, p. 132]{Miller-Mocanu-Book-2000-Diff-Sub}}]\label{Lemma-3.4h-p132-Miller-Mocanu}
	Let $q:\mathbb{D}\to\mathbb{C}$ be univalent, and let $\lambda$ and $\vartheta$ be analytic in a domain $\Omega\supseteq q(\mathbb{D})$ with $\lambda(\xi)\neq0$ whenever $\xi\in{q(\mathbb{D})}$. Define
	\begin{align*}
	\Theta(z):=zq'(z)\,\lambda(q(z)) \quad \text{ and } \quad  h(z):=\vartheta(q(z))+\Theta(z), \qquad z\in\mathbb{D}.
	\end{align*}
	Suppose that either
	\begin{enumerate}[\rm(i)]
		\item $h(z)$ is convex, or
		\item $\Theta(z)$ is starlike.\\
		In addition, assume that
		\item $\mathrm{Re}\left({zh'(z)}/{\Theta(z)}\right)>0$ in $\mathbb{D}$.
	\end{enumerate}
	If $p\in\mathcal{H}$ with $p(0)=q(0)$, $p(\mathbb{D})\subset{\Omega}$ and
	\begin{align*}
	\vartheta(p(z))+zp'(z)\,\lambda(p(z))\prec\vartheta(q(z))+zq'(z)\,\lambda(q(z)), \qquad z\in\mathbb{D},
	\end{align*}
	then $p\prec{q}$, and $q$ is the best dominant.
\end{lemma}

In the sequel, it is always assumed that $z\in\mathbb{D}$ unless stated otherwise.


\section{Subordination Results Related to Convex Domains}\label{Section-Lemniscate-Regions-SRN}
This section has been divided into two subsections. In the first one, $\mathcal{P}(\mathbb{D})$ is a lemniscate type convex domain and, in the second one, $\mathcal{P}(\mathbb{D})$ is almost circular.
\smallskip
\subsection{Lemniscate Type Domains} ~\\
Two different cases of $\mathcal{P}(z)$ are considered in this subsection. In the first theorem $\mathcal{P}(z):=\varphi_{\scriptscriptstyle{L}}(z)=\sqrt{1+z}$, the function which maps $\mathbb{D}$ onto the interior of the right-half of lemniscate of Bernoulli $\left(u^2+v^2\right)^2-2\left(u^2-v^2\right)=0$ (\Cref{Figure-LemB-SRN}), while in second theorem $\mathcal{P}(z):=\varphi_{\scriptscriptstyle{RL}}(z)=
\sqrt{2}-(\sqrt{2}-1)\sqrt{{(1-z)}/{((2\sqrt{2}-2)z+1)}}
$, 
the mapping of $\mathbb{D}$ onto the inside of the left-half of the shifted lemniscate of Bernoulli $\left((u-\sqrt{2})^2+v^2\right)^2-2\left((u-\sqrt{2})^2-v^2\right)=0$ (\Cref{Figure-SLemB-SRN}). The functions $\varphi_{\scriptscriptstyle{L}}(z)$ and $\varphi_{\scriptscriptstyle{RL}}(z)$ were introduced in \cite{Sokol-J.Stankwz-1996-Lem-of-Ber} and \cite{Mendiratta-2014-Shifted-Lemn-Bernoulli}, respectively.
\vspace{-1em}
\begin{figure}[H]	
	\begin{subfigure}{0.45\textwidth}
		\centering
		\includegraphics{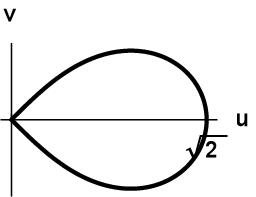}
		\caption{Boundary of $\varphi_{\scriptscriptstyle{L}}(\mathbb{D})$.}
		\label{Figure-LemB-SRN}
	\end{subfigure}
	\begin{subfigure}{0.45\textwidth}
		\centering
		\includegraphics{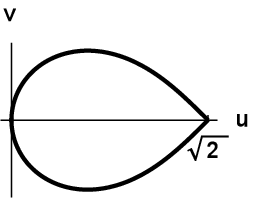}
		\caption{Boundary of $\varphi_{\scriptscriptstyle{RL}}(\mathbb{D})$.} 
		\label{Figure-SLemB-SRN}
	\end{subfigure}
	\caption{}
	\label{Figure-Small-Lemniscates}	
\end{figure}
 \vspace{-1em}
\begin{theorem}\label{Thrm-LemB-Impl-Neph-SRN}
Let $p\in\mathcal{H}$ satisfies $p(0)=1$, and let $\varphi_{\scriptscriptstyle{L}}(z):=\sqrt{1+z}$,
where the branch of the square root is chosen in order that $\varphi_{\scriptscriptstyle{L}}(0)=1$. Then each of the following subordinations imply $p(z)\prec\varphi_{\scriptscriptstyle {Ne}}(z):=1+z-z^3/3$.
\begin{enumerate}[\rm(a)]
\item
$1+\beta{zp'(z)}\prec\varphi_{\scriptscriptstyle{L}}(z)$ for $\beta\geq3(1-\log2)\approx0.920558$.
\item 
$1+\beta\left(\frac{zp'(z)}{p(z)}\right)\prec\varphi_{\scriptscriptstyle{L}}(z)$ for
              $\beta\geq\frac{2\left(\sqrt2+\log2-1-\log(1+\sqrt2)\right)}{\log(5/3)}\approx0.884792$.
\item 
$1+\beta\left(\frac{zp'(z)}{p^2(z)}\right)\prec\varphi_{\scriptscriptstyle{L}}(z)$ for
                                 $\beta\geq5\left(\sqrt2+\log2-1-\log(1+\sqrt2)\right)\approx1.12994$.
\end{enumerate}
Each estimate on $\beta$ is sharp.
\end{theorem}
\begin{proof}
\begin{enumerate}[(a):]
\item\label{Proof-Thrm-LemB-Impl-Neph-Part-a-SRN}
     Consider the first-order linear differential equation given by
     \begin{align}\label{Diff-Eqn-Lem-SRN}
     1+\beta{zq'_{\beta}(z)}=\varphi_{\scriptscriptstyle{L}}(z),
     \end{align}
     where $\varphi_{\scriptscriptstyle{L}}(z)$ is defined in the hypothesis. It is easy to verify that the analytic univalent function $q_{\beta}:\overline{\mathbb{D}}\to\mathbb{C}$ defined by
    \begin{align*}
    q_{\beta}(z)=1+\frac{2}{\beta}\Big(\varphi_{\scriptscriptstyle{L}}(z)+\log2-\log\left(1+\varphi_{\scriptscriptstyle{L}}(z)\right)-1\Big).
     \end{align*}
     is a solution of \eqref{Diff-Eqn-Lem-SRN}. For $\xi\in\mathbb{C}$, take $\vartheta(\xi)=1$ and $\lambda(\xi)=\beta$ in \Cref{Lemma-3.4h-p132-Miller-Mocanu} so that the functions $\Theta,h:\overline{\mathbb{D}}\to\mathbb{C}$ reduce to
       \begin{align*}
       \Theta(z)=zq'_{\beta}(z)\lambda\big(q_{\beta}(z)\big)={\beta}zq'_{\beta}(z)=\varphi_{\scriptscriptstyle{L}}(z)-1
      \end{align*}
       and
      \begin{align*}
      h(z)=\vartheta\big(q_\beta(z)\big)+\Theta(z)=1+\Theta(z)=\varphi_{\scriptscriptstyle{L}}(z).
      \end{align*}
     Since the image of $\mathbb{D}$ under the function $\varphi_{\scriptscriptstyle{L}}(z)$ is a convex domain, the function $h(z)$ is convex in $\mathbb{D}$. Further, as every convex function is starlike with respect to each of its points, the function $\Theta(z)=h(z)-1$ is starlike in $\mathbb{D}$. Therefore, by the analytic characterization of starlike functions, it follows that
     \begin{align*}
     \mathrm{Re}\left({zh'(z)}/{\Theta(z)}\right)=\mathrm{Re}\left({z\Theta'(z)}/{\Theta(z)}\right)>0, \quad z\in\mathbb{D}.
     \end{align*}
%
%
     Also $p(0)=1=q_{\beta}(0)$ shows that \Cref{Lemma-3.4h-p132-Miller-Mocanu} is applicable, and hence the differential subordination  
     \begin{align*}
     \vartheta(p(z))+zp'(z)&\lambda(p(z))=\\
     &1+{\beta}zp'(z)\prec\varphi_{\scriptscriptstyle{L}}(z)=1+{\beta}zq'_{\beta}(z)\\
     &\hspace{12em}=\vartheta(q_{\beta}(z))+zq'_{\beta}(z)\lambda(q_{\beta}(z))
     \end{align*}   
     implies the subordination $p\prec{q_{\beta}}$. Now, the desired result $p\prec\varphi_{\scriptscriptstyle{Ne}}$ will follow by the transitivity of $\prec$ if the subordination
    $q_{\beta}\prec\varphi_{\scriptscriptstyle{Ne}}$
    holds. The necessary condition for $q_{\beta}\prec\varphi_{\scriptscriptstyle{Ne}}$ to hold true is that
    \begin{align}\label{N&S-Cond-LemB-Impl-Neph-Part-a-SRN}
     {1}/{3}=\varphi_{\scriptscriptstyle {Ne}}(-1)<q_\beta(-1)<q_\beta(1)<\varphi_{\scriptscriptstyle {Ne}}(1)={5}/{3}.
     \end{align}
     On simplifying the condition \eqref{N&S-Cond-LemB-Impl-Neph-Part-a-SRN}, the following two  inequalities are obtained
     \begin{align*}
      \beta\geq3\left(1-\log2\right)=\beta_1\; \text{ and }\; \beta\geq3\left(\sqrt2+\log2-1-\log(1+\sqrt2)\right)=\beta_2.
        \end{align*}
       Therefore, the necessary condition for $q_{\beta}\prec\varphi_{\scriptscriptstyle{Ne}}$ is that $\beta\geq\max\{\beta_1,\beta_2\}=\beta_1=3\left(1-\log2\right)$. Moreover, a graphical observation (see \Cref{Fig1-LB-Impl-Neph-SRN}) shows that whenever $\beta\geq\beta_1=3\left(1-\log2\right)$, the range of $q_\beta(z)$ is completely contained in the nephroid domain $\varphi_{\scriptscriptstyle{Ne}}(\mathbb{D})$. Since the function $\varphi_{\scriptscriptstyle{Ne}}(z)$ is univalent in $\mathbb{D}$ and $q_\beta(0)=\varphi_{\scriptscriptstyle{Ne}}(0)=1$, we conclude that the condition $\beta\geq\beta_1$ sufficiently implies the subordination $q_{\beta}\prec\varphi_{\scriptscriptstyle{Ne}}$. The sharpness of the estimate on $\beta$ follows from the fact that $q_\beta(-1)=1/3$ for $\beta=3(1-\log2)$. This completes part \eqref{Proof-Thrm-LemB-Impl-Neph-Part-a-SRN} of \Cref{Thrm-LemB-Impl-Neph-SRN}.
       
      \begin{figure}[H]
	   \includegraphics[scale=0.75]{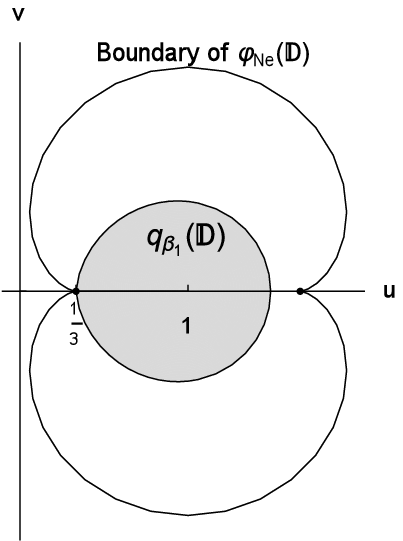}
	   \caption{For $\beta_1=3(1-\log2)$, $q_{\beta_1}(\mathbb{D})\subset\varphi_{\scriptscriptstyle{Ne}}(\mathbb{D})$.}
	    \label{Fig1-LB-Impl-Neph-SRN}	
        \end{figure}

\item\label{Proof-Thrm-LemB-Impl-Neph-Part-b-SRN}
      Define the analytic function $\hat{q}_\beta:\overline{\mathbb{D}}\to\mathbb{C}$ as
      \begin{align*}
      \hat{q}_{\beta}(z)
      =\exp\left(\frac{2}{\beta}\left(\varphi_{\scriptscriptstyle{L}}(z)+\log2-\log\left(1+\varphi_{\scriptscriptstyle{L}}(z)\right)-1\right)\right).
        \end{align*}
       The function $\hat{q}_\beta(z)$ satisfies the differential equation $1+\beta{z\hat{q}'_{\beta}(z)/\hat{q}_{\beta}(z)}=\varphi_{\scriptscriptstyle{L}}(z)$. On taking $\vartheta(\xi)=1$ and $\lambda(\xi)=\beta/\xi$ in \Cref{Lemma-3.4h-p132-Miller-Mocanu}, we obtain $\Theta(z)=z\hat{q}'_{\beta}(z)\lambda\big(\hat{q}_{\beta}(z)\big)
       =\beta{z\hat{q}'_{\beta}(z)/\hat{q}_{\beta}(z)}=\varphi_{\scriptscriptstyle{L}}(z)-1$ and $h(z)=1+\Theta(z)=\varphi_{\scriptscriptstyle{L}}(z)$. Again, the convexity of $\varphi_{\scriptscriptstyle{L}}(z)$ implies the convexity of $h(z)$ and the positiveness of $\mathrm{Re}\left(zh'(z)/\Theta(z)\right)=\mathrm{Re}\left(z\Theta'(z)/\Theta(z)\right)$ in $\mathbb{D}$. Applying \Cref{Lemma-3.4h-p132-Miller-Mocanu}, it follows that the first-order differential subordination
        \begin{align*}
       1+\beta\frac{zp'(z)}{p(z)}\prec\varphi_{\scriptscriptstyle{L}}(z)=1+\beta\frac{z\hat{q}'_{\beta}(z)}{\hat{q}_{\beta}(z)}
       \end{align*}
       implies $p\prec{\hat{q}_{\beta}}$. In view of this differential chain, the claimed subordination $p\prec\varphi_{\scriptscriptstyle{Ne}}$ holds if the subordination $\hat{q}_\beta\prec\varphi_{\scriptscriptstyle{Ne}}$ holds. Likewise in \eqref{Proof-Thrm-LemB-Impl-Neph-Part-a-SRN}, the subordination $\hat{q}_\beta\prec\varphi_{\scriptscriptstyle {Ne}}$ holds if, and only if, $1/3<\hat{q}_\beta(-1)<\hat{q}_\beta(1)<5/3$, which further gives
       \begin{align*}
        \beta\geq\frac{2(1-\log2)}{\log3}=\beta_1\; \text{ and }\; \beta\geq \frac{2\left(\sqrt2+\log2-1-\log(1+\sqrt2)\right)}{\log(5/3)}=\beta_2.
       \end{align*}
        Thus, the subordination $\hat{q}_{\beta}\prec\varphi_{\scriptscriptstyle{Ne}}$ holds provided $\beta\geq\max\{\beta_1,\beta_2\}=\beta_2$. Further, $\hat{q}_{\beta}(1)=5/3$ for $\beta=\beta_2$, showing that the bound on $\beta$ can not be decreased further, see \Cref{Fig2-LB-Impl-Neph-SRN} (the curve $\hat{\gamma}_\beta$).

\item\label{Proof-Thrm-LemB-Impl-Neph-Part-c-SRN}
        Consider the function $\tilde{q}_\beta$ defined on $\overline{\mathbb{D}}$ and given by       
         \begin{align*}
          \tilde{q}_\beta(z)
         =\left(1-\frac{2}{\beta}\left(\varphi_{\scriptscriptstyle{L}}(z)+\log2-\log\left(1+\varphi_{\scriptscriptstyle{L}}(z)\right)-1\right)\right)^{-1}.
         \end{align*}
         The function $\tilde{q}_\beta$ is analytic and is a solution of the first-order differential equation $1+\beta{z\tilde{q}'_\beta/\tilde{q}^2_\beta}=\varphi_{\scriptscriptstyle{L}}$. Let $\xi\in\mathbb{C}$. Setting $\vartheta(\xi)=1$ and $\lambda(\xi)=\beta/\xi^2$ in \Cref{Lemma-3.4h-p132-Miller-Mocanu}, the functions $\Theta(z)$ and $h(z)$ reduce to
        \begin{align*}
       \Theta(z)=z\tilde{q}'_\beta(z)\lambda(\tilde{q}_\beta(z))
        =\varphi_{\scriptscriptstyle{L}}(z)-1 \text{ and } h(z)=1+\Theta(z)=\varphi_{\scriptscriptstyle{L}}(z).
        \end{align*}
        As the function $h(z)=\varphi_{\scriptscriptstyle{L}}(z)=\sqrt{1+z}$ is convex in $\mathbb{D}$ and $\mathrm{Re}\left(zh'(z)/\Theta(z)\right)>0$ for each $z\in\mathbb{D}$, we conclude from \Cref{Lemma-3.4h-p132-Miller-Mocanu} that the first-order differential subordination
        \begin{align*}
        1+\beta\frac{zp'(z)}{p^2(z)}\prec\varphi_{\scriptscriptstyle{L}}(z)=1+\beta\frac{z\tilde{q}'_\beta(z)}{\tilde{q}^2_\beta(z)}
       \end{align*}
        implies the subordination $p{\prec}\tilde{q}_\beta$. To attain the subordination $p\prec\varphi_{\scriptscriptstyle{Ne}}$, we only need to show $\tilde{q}_\beta\prec\varphi_{\scriptscriptstyle{Ne}}$. As earlier, this is true if, and only if, $1/3<\tilde{q}_\beta(-1)<\tilde{q}_\beta(1)<5/3$. That is, if
         \begin{align*}
          \beta\geq\max&\left\{1-\log2,\,5\left(\sqrt2+\log2-1-\log(1+\sqrt2)\right)\right\}\\
                     &=5\left(\sqrt2+\log2-1-\log(1+\sqrt2)\right).
         \end{align*}
        For sharpness of the estimate obtained on the real $\beta$, verify that $\tilde{q}_\beta(1)=5/3$ when $\beta=5\left(\sqrt2-\log(1+\sqrt2)+\log2-1\right)$. Also see \Cref{Fig2-LB-Impl-Neph-SRN} (the curve $\tilde{\gamma}_\beta$). \qedhere
\end{enumerate}
\end{proof}

If $f\in\mathcal{A}$, then $p(z)=zf'(z)/f(z)\in\mathcal{H}$ and satisfies $p(0)=1$. In view of this observation, the following sufficient conditions for $\mathcal{S}^*_{Ne}$ are obtained on setting $p(z)=zf'(z)/f(z)$ in \Cref{Thrm-LemB-Impl-Neph-SRN}.
\begin{corollary}
	Let $f\in\mathcal{A}$, and let
	\begin{align}\label{Definition-J}
	\mathcal{G}(z):=1-\frac{zf'(z)}{f(z)}+\frac{zf''(z)}{f'(z)},\quad z\in\mathbb{D}.
	\end{align}
	 Then each of the following is sufficient to imply $f\in\mathcal{S}^*_{Ne}$.
	\begin{enumerate}[\rm(a)]
		\item $1+\beta\mathcal{G}(z)\left(\frac{zf'(z)}{f(z)}\right)\prec\varphi_{\scriptscriptstyle{L}}(z)$ for $\beta\geq 3(1-\log2)$,
		\item 
		$1+\beta\mathcal{G}(z)\prec\varphi_{\scriptscriptstyle{L}}(z)$ for $\beta\geq \frac{2\left(\sqrt2+\log2-\log(1+\sqrt2)-1\right)}{\log(5/3)}$,
		\item $1+\beta\mathcal{G}(z)\left(\frac{zf'(z)}{f(z)}\right)^{-1}\prec\varphi_{\scriptscriptstyle{L}}(z)$ for $\beta\geq5\left(\sqrt2+\log2-\log(1+\sqrt2)-1\right)$.
	\end{enumerate}
	The bounds on $\beta$ are best possible.
\end{corollary}

\begin{minipage}{0.45\textwidth}
  \centering
  \includegraphics[scale=0.85]{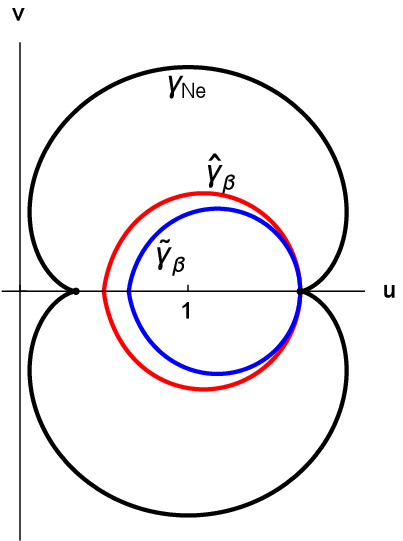}
\end{minipage}
\begin{minipage}{0.45\textwidth}
{\small{
   $\gamma_{Ne}:$ Boundary curve of $\varphi_{\scriptscriptstyle{Ne}}(\mathbb{D})$.\\
   \vspace{.01em}

     $\hat{\gamma}_\beta:$ Boundary curve of $\hat{q}_\beta(\mathbb{D})$ with\\ $\beta=\frac{2\left(\sqrt{2}-1+\log2-
     \log\left(1+\sqrt{2}\right)\right)}{\log\left(5/3\right)}$.\\
     \vspace{.01em}

     $\tilde{\gamma}_\beta:$ Boundary curve of $\tilde{q}_\beta(\mathbb{D})$
      with\\ $\beta=5\left(\sqrt{2}-1+\log2-\log\left(1+\sqrt{2}\right)\right)$.\\
    }}
\end{minipage}
\captionof{figure}{The functions $\hat{q}_\beta$ and $\tilde{q}_\beta$ are, respectively, the functions defined in
	 (\ref{Proof-Thrm-LemB-Impl-Neph-Part-b-SRN}) and
	(\ref{Proof-Thrm-LemB-Impl-Neph-Part-c-SRN}) of \Cref{Thrm-LemB-Impl-Neph-SRN}.}
\label{Fig2-LB-Impl-Neph-SRN}

\begin{theorem}\label{Thrm-ShftdLemB-Impl-Neph-SRN}
Let
\begin{align*}
\varphi_{\scriptscriptstyle{RL}}(z):=\sqrt{2}-(\sqrt{2}-1)\sqrt{\frac{1-z}{1+2(\sqrt{2}-1)z}}
\end{align*}
and
\begin{align*}
g_0(z):=\sqrt{2 \left(\sqrt{2}-1\right)}\times\tan ^{-1}\left(\frac{\sqrt{2 \left(\sqrt{2}-1\right)} \left(\sqrt{2 \left(\sqrt{2}-1\right) z+1}-\sqrt{1-z}\right)}{2 \left(\sqrt{2}-1\right) \sqrt{1-z}+\sqrt{2 \left(\sqrt{2}-1\right) z+1}}\right).
\end{align*}
Then, for $p\in\mathcal{H}$ satisfying $p(0)=1$, each of the following differential subordinations is sufficient to imply the subordination $p\prec\varphi_{\scriptscriptstyle {Ne}}$. Moreover, the respective bounds on $\beta$ can not be improved further.
\begin{enumerate}[\rm(a)]
\item
$1+\beta{zp'(z)}\prec\varphi_{\scriptscriptstyle{RL}}(z)$ for $\beta\geq
                 -\frac{3\left(2 \left(\sqrt{2}-1\right) \log \left(\frac{1}{2} \left(\sqrt{1-2 \left(\sqrt{2}-1\right)}+\sqrt{2}\right)\right)+g_0(-1)\right)}{2}\approx0.822832$.
\item
$1+\beta\left(\frac{zp'(z)}{p(z)}\right)\prec\varphi_{\scriptscriptstyle{RL}}(z)$ for $\beta\geq
            \frac{2 \left(\sqrt{2}-1\right) \log \left(\frac{1}{2} \sqrt{2 \left(\sqrt{2}-1\right)+1}\right)+g_0(1)}
                             {\log \left(\frac{5}{3}\right)}\approx0.680906$.
\item
$1+\beta\left(\frac{zp'(z)}{p^2(z)}\right)\prec\varphi_{\scriptscriptstyle{RL}}(z)$ for $\beta\geq
         \frac{5\left(2 \left(\sqrt{2}-1\right) \log \left(\frac{1}{2} \sqrt{2 \left(\sqrt{2}-1\right)+1}\right)
         +g_0(1)\right)}{2}\approx0.869561$.
\end{enumerate}
\end{theorem}
\begin{proof}
\begin{enumerate}[(a):]
\item
      Here, we consider the first-order linear differential equation $1+\beta{zq'_{\beta}(z)}=\varphi_{\scriptscriptstyle{RL}}(z)$. This differential equation has an analytic solution $q_{\beta}$ defined on $\overline{\mathbb{D}}$ given by
      \begin{align*}
      q_{\beta}(z)=1+\frac{1}{\beta}\left(2\left(\sqrt{2}-1\right)
      \log\left(\frac{\Psi_{\scriptscriptstyle{RL}}(z)}{2}\right)+g_0(z)\right),
      \end{align*}
     where $g_0(z)$ is defined in the hypothesis and 
      \begin{align}\label{Def-Psi-RL-SRN}
      \Psi_{\scriptscriptstyle{RL}}(z):=\sqrt{1-z}+\sqrt{2\left(\sqrt{2}-1\right)z+1}.
      \end{align}      
     Let $\xi\in\mathbb{C}$. On choosing $\vartheta(\xi)=1$ and $\lambda(\xi)=\beta$ in \Cref{Lemma-3.4h-p132-Miller-Mocanu}, we obtain $\Theta(z)=zq'_{\beta}(z)\lambda(q_{\beta}(z))=\varphi_{\scriptscriptstyle{RL}}(z)-1$ and $h(z)=1+\Theta(z)=\varphi_{\scriptscriptstyle{RL}}(z)$. Since the function $\varphi_{\scriptscriptstyle{RL}}(z)$ sends $\mathbb{D}$ onto a convex region, the function $h$ is convex. Moreover, $h$ satisfies $\mathrm{Re}\left(zh'(z)/\Theta(z)\right)>0$ for each $z\in\mathbb{D}$. An application of \Cref{Lemma-3.4h-p132-Miller-Mocanu} shows that the first-order differential subordination $1+\beta{zp'(z)}\prec\varphi_{\scriptscriptstyle{RL}}(z)=1+\beta{zq'_{\beta}(z)}$ yields the subordination $p\prec{q_{\beta}}$. The required subordination $p\prec\varphi_{\scriptscriptstyle{Ne}}$ will now follow by showing $q_\beta\prec\varphi_{\scriptscriptstyle{Ne}}$. If $q_\beta\prec\varphi_{\scriptscriptstyle{Ne}}$, then
     \begin{align}\label{N&S-Cond-ShftdLemB-Impl-Neph-Part-a-SRN}
     \frac{1}{3}=\varphi_{\scriptscriptstyle {Ne}}(-1)<q_\beta(-1)<q_\beta(1)<\varphi_{\scriptscriptstyle {Ne}}(1)=\frac{5}{3}.
     \end{align}
     Fortunately, the condition \eqref{N&S-Cond-ShftdLemB-Impl-Neph-Part-a-SRN} turns out to be sufficient for the subordination $q_\beta\prec\varphi_{\scriptscriptstyle{Ne}}$ to hold, as can be seen from the graphs of the respective functions. Since, $q_{\beta}(-1)\geq{1}/{3}$ whenever
     \begin{align*}
       \beta\geq -\frac{3}{2}\left(2\left(\sqrt{2}-1\right)
            \log\left(\frac{\Psi_{\scriptscriptstyle{RL}}(-1)}{2}\right)+g_0(-1)\right)=\beta_1
     \end{align*}
     and $q_{\beta}(1)\leq{5}/{3}$ whenever
     \begin{align*}
     \beta\geq
       \frac{3}{2}\left(2\left(\sqrt{2}-1\right)
       \log\left(\frac{\Psi_{\scriptscriptstyle{RL}}(1)}{2}\right)+g_0(1)\right)=\beta_2,
       \end{align*}
       it follows from \eqref{N&S-Cond-ShftdLemB-Impl-Neph-Part-a-SRN} that the subordination $q_{\beta}\prec\varphi_{\scriptscriptstyle{Ne}}$ holds true whenever
       $\beta\geq\max\{\beta_1,\beta_2\}=\beta_1$. Also, the value of $q_{\beta}(-1)$ at $\beta=\beta_1$ is $1/3$. This proves that the estimate on $\beta$ is sharp.
\item
       Let $\Psi_{\scriptscriptstyle{RL}}(z)$ be given as in \eqref{Def-Psi-RL-SRN}. Then, elementary analysis shows that the function $\hat{q}_\beta$ given by       
       \begin{align*}
       \hat{q}_{\beta}(z)=\exp\left(\frac{2\left(\sqrt{2}-1\right)\log\left(\frac{\Psi_{\scriptscriptstyle{RL}}(z)}{2}\right)+g_0(z)}{\beta}\right)
       \end{align*}
        is analytic in $\overline{\mathbb{D}}$ and satisfies $1+\beta{z\hat{q}'_\beta(z)/\hat{q}_\beta(z)}=\varphi_{\scriptscriptstyle{RL}}(z)$. Defining the functions $\vartheta$ and $\lambda$ likewise in \Cref{Thrm-LemB-Impl-Neph-SRN}\eqref{Proof-Thrm-LemB-Impl-Neph-Part-b-SRN}, we see that the function $h(z)=\varphi_{\scriptscriptstyle{RL}}(z)$ is convex and $\mathrm{Re}\left(zh'/\Theta\right)>0$ in $\mathbb{D}$. Hence, from \Cref{Lemma-3.4h-p132-Miller-Mocanu}, it follows that
        \begin{align*}
        1+\beta\frac{zp'(z)}{p(z)}\prec\varphi_{\scriptscriptstyle{RL}}(z)=1+\beta\frac{z\hat{q}'_\beta(z)}{\hat{q}_\beta(z)}
        \end{align*}
        implies the subordination $p\prec{\hat{q}_\beta}$. Now, to arrive at the subordination $p\prec\varphi_{\scriptscriptstyle{Ne}}$, it is required that the subordination $\hat{q}_\beta\prec\varphi_{\scriptscriptstyle {Ne}}$ should hold. As in \Cref{Thrm-LemB-Impl-Neph-SRN}(\ref{Proof-Thrm-LemB-Impl-Neph-Part-b-SRN}), $\hat{q}_\beta\prec\varphi_{\scriptscriptstyle {Ne}}$ if $\beta\geq\max\{\beta_1,\beta_2\}$, where
       \begin{align*}
        \beta_1=-\frac{2\left(\sqrt{2}-1\right)\log \left(\frac{
        		\Psi_{\scriptscriptstyle{RL}}(-1)
        	}{2}\right)+g_0(-1)}{\log3}
       \end{align*}
      and
      \begin{align*}
      \beta_2=\frac{2 \left(\sqrt{2}-1\right) \log \left(\frac{
      		\Psi_{\scriptscriptstyle{RL}}(1)
      	}{2}\right)+g_0(1)}{\log(5/3)}.
      \end{align*}
\item
       Let $\Psi_{\scriptscriptstyle{RL}}(z)$ be defined as in \eqref{Def-Psi-RL-SRN}. Verify that the function
\begin{align*}
\tilde{q}_\beta(z)=\left(1-\frac{1}{\beta}\left(2\left(\sqrt{2}-1\right)
      \log\left(\frac{
      	\Psi_{\scriptscriptstyle{RL}}(z)
      }{2}\right)+g_0(z)\right)\right)^{-1}
\end{align*}
is an analytic solution of the first-order differential equation $1+\beta{z\tilde{q}'_\beta/\tilde{q}_\beta^2}=\varphi_{\scriptscriptstyle{RL}}$. On defining $\vartheta$ and $\lambda$ as in \Cref{Thrm-LemB-Impl-Neph-SRN}\eqref{Proof-Thrm-LemB-Impl-Neph-Part-c-SRN}, we get the functions $\Theta$ and $h$ defined in \Cref{Lemma-3.4h-p132-Miller-Mocanu} as $\Theta=\varphi_{\scriptscriptstyle{RL}}-1$ and $h=1+\Theta=\varphi_{\scriptscriptstyle{RL}}$. Again, the function $h(z)=\varphi_{\scriptscriptstyle{RL}}(z)$ is convex and $\mathrm{Re}\left(zh'(z)/h(z)\right)$ is positive in $\mathbb{D}$, so that from \Cref{Lemma-3.4h-p132-Miller-Mocanu} we have the implication:
\begin{align*}
1+\beta\frac{zp'(z)}{p^2(z)}\prec1+\beta\frac{z\tilde{q}'_\beta(z)}{\tilde{q}^2_\beta(z)}\implies p(z)\prec{\tilde{q}_\beta(z)}.
\end{align*}
Now it suffices to prove $\tilde{q}_\beta\prec\varphi_{\scriptscriptstyle{Ne}}$. As in \Cref{Thrm-LemB-Impl-Neph-SRN}(\ref{Proof-Thrm-LemB-Impl-Neph-Part-c-SRN}), $\tilde{q}_\beta\prec\varphi_{\scriptscriptstyle{Ne}}$ whenever $\beta\geq\max\{\beta_1,\beta_2\}$, where
\begin{align*}
\beta_1=-\frac{1}{2}\left(\left(2\sqrt{2}-2\right)\log \left(\frac{\Psi_{\scriptscriptstyle{RL}}(-1)}{2}\right)+g_0(-1)\right)
\end{align*}
and
\begin{align*}
\beta_2=\frac{5}{2} \left(\left(2\sqrt{2}-2\right) \log \left(\frac{\Psi_{\scriptscriptstyle{RL}}(1)}{2}\right)+g_0(1)\right).
\tag*{\qedhere}
\end{align*}
\end{enumerate}
\end{proof}

As in the previous theorem, the following sufficient conditions for the function class $\mathcal{S}^*_{Ne}$ immediately follow from \Cref{Thrm-ShftdLemB-Impl-Neph-SRN}.
\begin{corollary}
	Let $f\in\mathcal{A}$ and let $\mathcal{G}(z)$ be given by \eqref{Definition-J}. Then each of the following conditions imply $f\in\mathcal{S}^*_{Ne}$.
	\begin{enumerate}[\rm(a)]
		\item $1+\beta\mathcal{G}(z)\left(\frac{zf'(z)}{f(z)}\right)\prec\varphi_{\scriptscriptstyle{RL}}(z)$ for $\beta\geq-\frac{3\left(2 \left(\sqrt{2}-1\right) \log \left(\frac{1}{2} \left(\sqrt{1-2 \left(\sqrt{2}-1\right)}+\sqrt{2}\right)\right)+g_0(-1)\right)}{2}$,
		\item
		$1+\beta{\mathcal{G}(z)}\prec\varphi_{\scriptscriptstyle{RL}}(z)$ for $\beta\geq\frac{2 \left(\sqrt{2}-1\right) \log \left(\frac{1}{2} \sqrt{2 \left(\sqrt{2}-1\right)+1}\right)+g_0(1)}
		{\log \left(\frac{5}{3}\right)}$,
		\item $1+\beta\mathcal{G}(z)\left(\frac{zf'(z)}{f(z)}\right)^{-1}\prec\varphi_{\scriptscriptstyle{RL}}(z)$ for $\beta\geq\frac{5\left(2 \left(\sqrt{2}-1\right) \log \left(\frac{1}{2} \sqrt{2 \left(\sqrt{2}-1\right)+1}\right)
			+g_0(1)\right)}{2}$.
	\end{enumerate}
	The bounds on $\beta$ are sharp.
\end{corollary}
\smallskip
\subsection{Other Convex Domains} ~\\ 
In this subsection, $\mathcal{P}(z)$ is either the modified sigmoid function $2/(1+e^{-z})$ (see \cite{Goel-Siva-2019-Sigmoid-BMMS}), or the exponential function $e^z$ \cite{Mendiratta-Ravi-2015-Expo-BMMS}.

\begin{theorem}
	Let $\varphi_{\scriptscriptstyle{SG}}(z):={2}/{(1+e^{-z})}$, and let
	\begin{align*}
	\ell(z)=\int_0^z \frac{e^t-1}{t \left(e^t+1\right)} \, dt.
	\end{align*}
	Then, for $p\in\mathcal{H}$ with $p(0)=1$, each of the following differential subordinations is sufficient for the subordination $p\prec\varphi_{\scriptscriptstyle{Ne}}$:
	\begin{enumerate}[\rm(a)]
		\item 
		$1+\beta{zp'}\prec\varphi_{\scriptscriptstyle{SG}}(z)$ for $\beta\geq {3}\ell(1)/2\approx0.730333$,
		\item 
		$1+\beta{zp'}/{p}\prec\varphi_{\scriptscriptstyle{SG}}(z)$ for $\beta\geq {\ell(1)}/{\log(5/3)}\approx0.953141$,
		\item 
		$1+\beta{zp'}/{p^2}\prec\varphi_{\scriptscriptstyle{SG}}(z)$ for $\beta\geq {5}\ell(1)/2\approx1.21722$.
	\end{enumerate}
	The bounds on $\beta$ can not be improved further.
\end{theorem}
\begin{proof}
	\begin{enumerate}[(a):]
		\item 
		A simple analysis shows that the analytic function $q_{\beta}:\overline{\mathbb{D}}\to\mathbb{C}$ given by
		\begin{align*}
		q_{\beta}(z)=1+\frac{1}{\beta}\left(\frac{z}{2}-\frac{z^3}{72}+\frac{z^5}{1200}-\frac{17 z^7}{282240}+\frac{31 z^9}{6531840}+\cdots\right),
		\end{align*}
		satisfies the first-order linear differential equation $1+\beta{zq'_{\beta}}=\varphi_{\scriptscriptstyle{SG}}$. Defining the functions $\vartheta$ and $\lambda$ as in \Cref{Thrm-LemB-Impl-Neph-SRN}\eqref{Proof-Thrm-LemB-Impl-Neph-Part-a-SRN}, we find that   
%
		$\Theta(z)=zq'_{\beta}(z)\lambda(q_{\beta}(z))=(e^z-1)/(e^z+1)$ is starlike and $\mathrm{Re}\left(zh'/\Theta\right)$ is positive in $\mathbb{D}$. Hence, \Cref{Lemma-3.4h-p132-Miller-Mocanu} says that $1+\beta{zp'}\prec\varphi_{\scriptscriptstyle{SG}}$ implies the subordination $p\prec{q_{\beta}}$. To arrive at the desired result, we only need to prove that the subordination $q_{\beta}\prec\varphi_{\scriptscriptstyle {Ne}}$ holds. Proceeding as in \Cref{Thrm-LemB-Impl-Neph-SRN}(\ref{Proof-Thrm-LemB-Impl-Neph-Part-a-SRN}), we see that this is true if $\beta\geq 3\ell(1)/2$.
		
		\item 
		Here, consider the differential equation $1+\beta{z\hat{q}'_{\beta}(z)}/\hat{q}(z)=\varphi_{\scriptscriptstyle{SG}}(z)$ along with its analytic solution 
		\begin{align*}
		\hat{q}_{\beta}(z)=\exp\left(\frac{1}{\beta}\left(\frac{z}{2}-\frac{z^3}{72}+\frac{z^5}{1200}-\frac{17 z^7}{282240}+\frac{31 z^9}{6531840}+\cdots\right)\right),
		\end{align*}
		defined on $\overline{\mathbb{D}}$. Now to establish the subordination $p\prec\varphi_{\scriptscriptstyle{Ne}}$, continue as  \Cref{Thrm-LemB-Impl-Neph-SRN}(\ref{Proof-Thrm-LemB-Impl-Neph-Part-b-SRN}).
		\item 
		Taking the function $\tilde{q}_\beta$ as
		\begin{align*}
		\tilde{q}_{\beta}(z)=\left(1-\frac{1}{\beta}\left(\frac{z}{2}-\frac{z^3}{72}+\frac{z^5}{1200}-\frac{17 z^7}{282240}+\frac{31 z^9}{6531840}+\cdots\right)\right)^{-1},
		\end{align*}
		and following \Cref{Thrm-LemB-Impl-Neph-SRN}(\ref{Proof-Thrm-LemB-Impl-Neph-Part-c-SRN}) yields the desired conclusion.
		\qedhere
	\end{enumerate}
\end{proof}

\begin{theorem}\label{Thrm-Exponential-Impl-Neph-SRN}
	Let $\varphi_{\scriptscriptstyle{e}}(z):=e^{z}$ be the exponential function, and let $p(z)$ be analytic such that $p(0)=1$. If any one of the following differential subordinations hold true, then $p\prec\varphi_{\scriptscriptstyle{Ne}}$. Each estimate on $\beta$ is sharp.
	\begin{enumerate}[\rm(a)]
		\item 
		$1+\beta{zp'}\prec\varphi_{\scriptscriptstyle{e}}(z)$ for $\beta\geq \sum _{n=1}^{\infty }\frac{3}{2n(n!)}\approx1.97685$,
		\item 
		$1+\beta{zp'}/{p}\prec\varphi_{\scriptscriptstyle{e}}(z)$ for $\beta\geq\frac{\sum _{n=1}^{\infty } \frac{1}{n (n!)}}{\log \left({5}/{3}\right)}\approx2.57995$,
		\item 
		$1+\beta{zp'}/{p^2}\prec\varphi_{\scriptscriptstyle{e}}(z)$ for $\beta\geq \sum _{n=1}^{\infty } \frac{5}{2n(n!)}\approx3.29476$.
	\end{enumerate}
\end{theorem}
\begin{proof}
	\begin{enumerate}[(a):]
		\item 
		Let the function $q_{\beta}$ be given by 
		\begin{align*}
		q_{\beta}(z)=1+\frac{1}{\beta}\sum _{n=1}^{\infty } \frac{z^n}{n (n!)}.
		\end{align*}
		This function is analytic on $\overline{\mathbb{D}}$ and satisfies the differential equation $1+\beta{zq'_{\beta}}=\varphi_{\scriptscriptstyle{e}}$. Noting that $e^{z}-1$ is starlike in $\mathbb{D}$ and proceeding as in \Cref{Thrm-LemB-Impl-Neph-SRN}(\ref{Proof-Thrm-LemB-Impl-Neph-Part-a-SRN}), the subordination $p\prec\varphi_{\scriptscriptstyle{Ne}}$ can be easily established.\\
		For $(b)$ and $(c)$, proceed as in \Cref{Thrm-LemB-Impl-Neph-SRN}(\ref{Proof-Thrm-LemB-Impl-Neph-Part-b-SRN}) and \Cref{Thrm-LemB-Impl-Neph-SRN}(\ref{Proof-Thrm-LemB-Impl-Neph-Part-c-SRN}), respectively.\qedhere
	\end{enumerate}
\end{proof}


\section{Subordination Results Related to non-Convex Domains} \label{Section-Non-Convex-Cusped-Regions-SRN}
Here we take $\mathcal{P}(z)$ to be any of the following Carath\'{e}odory functions:
\begin{enumerate}[(i)]
	\item
	$\varphi_{\scriptscriptstyle{\leftmoon}}(z)=z+\sqrt{1+z^2}$, which was introduced in 
	\cite{Gandhi-Ravi-2017-Lune, Raina-Sokol-2015-Crescent-Shaped-I} 
	and maps $\mathbb{D}$ onto a {\it crescent} shaped region (\Cref{Figure-Crescent-SRN}).
	\item
	$\varphi_{\scriptscriptstyle{C}}(z)=1+4z/3+2z^2/3$. The function $\varphi_{\scriptscriptstyle{C}}(z)$, introduced in \cite{Sharma-Ravi-2016-Cardioid}, maps $\partial\mathbb{D}$ onto the {\it cardioid} $(9u^2+9v^2-18u+5)^2-16(9u^2+9v^2-6u+1)=0$ (\Cref{Figure-Cardioid-SRN}).
	\item
	$\varphi_{\scriptscriptstyle{0}}(z)=1+\frac{z}{k}\left(\frac{k+z}{k-z}\right)$ with $k=1+\sqrt{2}$. The rational function $\varphi_{\scriptscriptstyle{0}}(z)$ was introduced by Kumar and Ravichandran \cite{Kumar-Ravi-2016-Starlike-Associated-Rational-Function}, and the region $\varphi_{\scriptscriptstyle{0}}(\mathbb{D})$ is the interior of a shifted cardioid.
	\item
	$\varphi_{\scriptscriptstyle{S}}(z)=1+\sin{z}$, which maps $\partial\mathbb{D}$ onto an {\it eight-shaped curve}. This function was introduced by Cho et al. \cite{Cho-2019-Sine-BIMS} (\Cref{Figure-SineFn-SRN}).	
\end{enumerate}
\begin{figure}[H]	
	\begin{subfigure}{0.3\textwidth}
		\centering
		\includegraphics{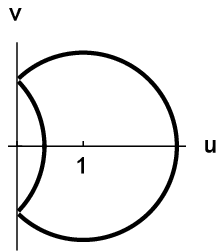}
		\caption{Boundary of $\varphi_{\scriptscriptstyle{\leftmoon}}(\mathbb{D})$.}
		\label{Figure-Crescent-SRN}
	\end{subfigure}
	\begin{subfigure}{0.3\textwidth}
		\centering
		\includegraphics{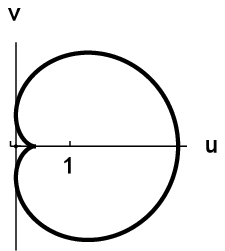}
		\caption{Boundary of $\varphi_{\scriptscriptstyle{C}}(\mathbb{D})$.} 
		\label{Figure-Cardioid-SRN}
	\end{subfigure}
\begin{subfigure}{0.3\textwidth}
	\centering
	\includegraphics{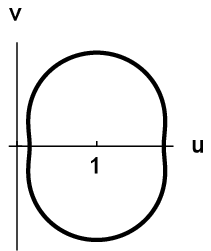}
	\caption{Boundary of $\varphi_{\scriptscriptstyle{S}}(\mathbb{D})$.} 
	\label{Figure-SineFn-SRN}
\end{subfigure}
	\caption{}
	\label{Figure-Small-Non-Convex-Cusped-Domains}	
\end{figure}
\begin{theorem}\label{Thrm-Crescent-Impl-Neph-SRN}
	Let $p\in\mathcal{H}$ with $p(0)=1$, and let $\varphi_{\scriptscriptstyle{\leftmoon}}(z):=z+\sqrt{1+z^2}$. If any of the following differential subordinations hold true, then $p\prec\varphi_{\scriptscriptstyle {Ne}}$.
	\begin{enumerate}[\rm(a)]
		\item 
		$1+\beta{zp'(z)}\prec\varphi_{\scriptscriptstyle{\leftmoon}}(z)$ for $\beta\geq\frac{3\left(\sqrt{2}-\log \left(1+\sqrt{2}\right)+\log2\right)}{2}\approx1.83898$,
		\item 
		$1+\beta\left(\frac{zp'(z)}{p(z)}\right)\prec\varphi_{\scriptscriptstyle{\leftmoon}}(z)$ for $\beta\geq\frac{\sqrt{2}+\log (2)-\log \left(1+\sqrt{2}\right)}{\log \left(\frac{5}{3}\right)}\approx2.40001$,
		\item
		 $1+\beta\left(\frac{zp'(z)}{p^2(z)}\right)\prec\varphi_{\scriptscriptstyle{\leftmoon}}(z)$ for $\beta\geq\frac{5\left(\sqrt{2}-\log \left(1+\sqrt{2}\right)+\log2\right)}{2}\approx3.06497$.
	\end{enumerate}
	Each estimate is sharp.
\end{theorem}
\begin{proof}
	\begin{enumerate}[(a):]
\item \label{Proof-Thrm-Crescent-Impl-Neph-Part-a-SRN}
		Consider the function $q_{\beta}$ defined on $\overline{\mathbb{D}}$ as
		\begin{align*}
		 q_{\beta}(z)=1+\frac{1}{\beta}\Big(\varphi_{\scriptscriptstyle{\leftmoon}}(z)-\log\left(1+\sqrt{1+z^2}\right)+\log2-1\Big).
		\end{align*}
		The function $q_{\beta}$ is an analytic solution of $1+\beta{zq'_{\beta}(z)}=\varphi_{\scriptscriptstyle{\leftmoon}}(z)$. For $\xi\in\mathbb{C}$, define $\vartheta(\xi)=1$ and $\lambda(\xi)=\beta$ to obtain the functions $\Theta$ and $h$ given in \Cref{Lemma-3.4h-p132-Miller-Mocanu} as  $\Theta(z)=\varphi_{\scriptscriptstyle{\leftmoon}}(z)-1$ and $h(z)=\varphi_{\scriptscriptstyle{\leftmoon}}(z)$. Since $\varphi_{\scriptscriptstyle{\leftmoon}}(\mathbb{D})$ is a region starlike with respect to $1$, the function $\Theta(z)=\varphi_{\scriptscriptstyle{\leftmoon}}(z)-1$ is starlike (w. r. t. origin) in $\mathbb{D}$. Further, the analytic characterization of the starlikeness of $\Theta$ implies that $\mathrm{Re}\left(zh'(z)/\Theta(z)\right)=\mathrm{Re}\left(z\Theta'(z)/\Theta(z)\right)$ is positive in $\mathbb{D}$. Therefore, an application of \Cref{Lemma-3.4h-p132-Miller-Mocanu} yields that the differential subordination
		$1+\beta{zp'(z)}\prec\varphi_{\scriptscriptstyle{\leftmoon}}(z)=1+\beta{zq'_{\beta}(z)}$ implies the subordination $p\prec{q_{\beta}}$. Now, our result $p\prec\varphi_{\scriptscriptstyle {Ne}}$ will follow if the subordination $q_{\beta}\prec\varphi_{\scriptscriptstyle {Ne}}$ holds. The necessary condition for the subordination $q_\beta\prec\varphi_{\scriptscriptstyle{Ne}}$ to hold true is that
		\begin{align}\label{N&S-Cond-Crscnt-Impl-Neph-Part-a-SRN}
		\varphi_{\scriptscriptstyle {Ne}}(-1)<q_\beta(-1)<q_\beta(1)<\varphi_{\scriptscriptstyle {Ne}}(1).
		\end{align}
		Simplifying the condition \eqref{N&S-Cond-Crscnt-Impl-Neph-Part-a-SRN}, we obtain the following two inequalities
		\begin{align*}
		\beta\geq\frac{3\left(2-\sqrt{2}+\log\left(1+\sqrt{2}\right)-\log2\right)}{2}=\beta_1
		\end{align*}
		and
		\begin{align*}
		\beta\geq\frac{3\left(\sqrt{2}-\log \left(1+\sqrt{2}\right)+\log2\right)}{2}=\beta_2.
		\end{align*}
	    Thus, for the subordination $q_\beta\prec\varphi_{\scriptscriptstyle{Ne}}$ to hold true, it is necessary that $\beta\geq\max\{\beta_1,\beta_2\}=\beta_2$. Moreover, the image of $\mathbb{D}$ under the function $q_\beta$ completely lies in the interior of the region bounded by the nephroid \eqref{Equation-of-Nephroid} whenever $\beta\geq\beta_2$, see \Cref{Fig-Lune-Impl-Neph-SRN} (curve 1). Now, the univalency of the function $\varphi_{\scriptscriptstyle{Ne}}(z)$ leads us to conclude that $q_\beta\prec\varphi_{\scriptscriptstyle{Ne}}$ if, and only if, $\beta\geq\beta_2$. Moreover, for $\beta=\beta_2$, the value of $q_\beta(z)$ at $z=1$ is $5/3$. This shows that the estimate on $\beta$ can not be decreased further.	
\item \label{Proof-Thrm-Crescent-Impl-Neph-Part-b-SRN}
		Clearly the analytic function
		\begin{align*}
		\hat{q}_\beta(z)
		 =\exp\left(\frac{1}{\beta}\Big(\varphi_{\scriptscriptstyle{\leftmoon}}(z)-1-\log\left(1+\sqrt{1+z^2}\right)+\log2\Big)\right)
		\end{align*}
		satisfies $1+\beta{z\hat{q}'_\beta/\hat{q}_\beta}=\varphi_{\scriptscriptstyle{\leftmoon}}(z)$. Let $\vartheta(\xi)=1$ and $\lambda(\xi)=\beta/\xi$, so that $\Theta(z)=\varphi_{\scriptscriptstyle{\leftmoon}}(z)-1$ and $h(z)=\varphi_{\scriptscriptstyle{\leftmoon}}(z)$. Thus, $\Theta$ is starlike and $\mathrm{Re}\left(zh'(z)/\Theta(z)\right)>0$. In view of \Cref{Lemma-3.4h-p132-Miller-Mocanu}, we have
		\begin{align*}
		1+\beta{zp'}/{p}\prec1+\beta{z\hat{q}'_\beta}/{\hat{q}_\beta} \implies p\prec{\hat{q}_\beta}.
		\end{align*} 		 
         To prove $p\prec\varphi_{\scriptscriptstyle{Ne}}$, it only remains to show that $\hat{q}_\beta\prec\varphi_{\scriptscriptstyle{Ne}}$. The later subordination is true if, and only if, $1/3<\hat{q}_\beta(-1)<\hat{q}_\beta(1)<5/3$. This condition on further simplification shows that $\hat{q}_\beta\prec\varphi_{\scriptscriptstyle{Ne}}$ provided 
		\begin{align*}
		\beta\geq\max&\left\{-\frac{\sqrt{2}-2+\log2-\log \left(1+\sqrt{2}\right)}{\log3},
		\frac{\sqrt{2}+\log2-\log\left(1+\sqrt{2}\right)}{\log (5/3)}\right\}\\
		&=\frac{\sqrt{2}+\log2-\log \left(1+\sqrt{2}\right)}{\log (5/3)}=\beta_0.
		\end{align*}
		Furthermore, a simple verification shows that $\hat{q}_\beta(1)=5/3$ for $\beta=\beta_0$. This proves that the lower bound $\beta_0$ on $\beta$ is sharp. See \Cref{Fig-Lune-Impl-Neph-SRN} (curve 2).
		
\item \label{Proof-Thrm-Crescent-Impl-Neph-Part-c-SRN}
		The function $\tilde{q}_\beta$ defined on $\overline{\mathbb{D}}$ by
		\begin{align*}
		\tilde{q}_\beta(z)
		 =\left(1-\frac{1}{\beta}\left(\varphi_{\scriptscriptstyle{\leftmoon}}(z)-1-\log\left(1+\sqrt{1+z^2}\right)+\log2\right)\right)^{-1}
		\end{align*}
		is an analytic solution of $1+\beta{z\tilde{q}'_\beta/\tilde{q}^2_\beta}=\varphi_{\scriptscriptstyle{\leftmoon}}(z)$. Defining $\vartheta(\xi)=1$ and $\lambda(\xi)=\beta/\xi^2$ we have $\Theta(z)=\varphi_{\scriptscriptstyle{\leftmoon}}(z)-1$ and $h(z)=\varphi_{\scriptscriptstyle{\leftmoon}}(z)$. Using the geometric properties of $\varphi_{\scriptscriptstyle{\leftmoon}}(z)$, we see that the conditions in the hypothesis of \Cref{Lemma-3.4h-p132-Miller-Mocanu} are satisfied. Therefore, the first-order differential subordination $1+\beta{zp'}/{p^2}\prec1+\beta{z\tilde{q}'_\beta}/{\tilde{q}^2_\beta}$ implies the subordination $p\prec{\tilde{q}_\beta}$.     		
		The result $p\prec\varphi_{\scriptscriptstyle{Ne}}$ will now follow by showing $\tilde{q}_\beta\prec\varphi_{\scriptscriptstyle{Ne}}$, which is true if, and only if, $1/3<\tilde{q}_\beta(-1)<\tilde{q}_\beta(1)<5/3$. As earlier, this condition gets satisfied if $\beta\geq\max\left\{\beta_1,\beta_2\right\}=\beta_2$, where
		\begin{align*}
		\beta_1=-\frac{\left(\sqrt{2}-2-\log \left(1+\sqrt{2}\right)+\log2\right)}{2}
		\end{align*}
		and
		\begin{align*}
		\beta_2=\frac{5\left(\sqrt{2}-\log \left(1+\sqrt{2}\right)+\log2\right)}{2}.
		\end{align*}
		The fact that $\tilde{q}_{\beta_2}(1)=5/3$ proves that the value $\beta=\beta_2$ is best possible, see \Cref{Fig-Lune-Impl-Neph-SRN} (curve 3).
\qedhere
	\end{enumerate}
\end{proof}
On fixing $p(z)=zf'(z)/f(z)$ in \Cref{Thrm-Crescent-Impl-Neph-SRN}, the following sufficient conditions for $\mathcal{S}^*_{Ne}$ follow.
\begin{corollary}
	Let $f\in\mathcal{A}$ and let $\mathcal{G}(z)$ be defined by \eqref{Definition-J}. Then each of the following conditions sufficiently implies that $f$ is a member of $\mathcal{S}^*_{Ne}$.
	\begin{enumerate}[\rm(a)]
		\item 
		$1+\beta\mathcal{G}(z)({zf'}/{f})\prec\varphi_{\scriptscriptstyle{\leftmoon}}(z)$ for $\beta\geq\frac{3\left(\sqrt{2}-\log \left(1+\sqrt{2}\right)+\log2\right)}{2}$,
		\item 
		$1+\beta{\mathcal{G}(z)}\prec\varphi_{\scriptscriptstyle{\leftmoon}}(z)$ for $\beta\geq\frac{\sqrt{2}+\log (2)-\log \left(1+\sqrt{2}\right)}{\log \left(\frac{5}{3}\right)}$,
		\item
		$1+\beta\mathcal{G}(z)\left({zf'}/{f}\right)^{-1}\prec\varphi_{\scriptscriptstyle{\leftmoon}}(z)$ for $\beta\geq\frac{5\left(\sqrt{2}-\log \left(1+\sqrt{2}\right)+\log2\right)}{2}$.
	\end{enumerate}
	Each estimate on $\beta$ is sharp.
\end{corollary}

\begin{minipage}{0.5\textwidth}
	\centering
	\includegraphics[scale=0.85]{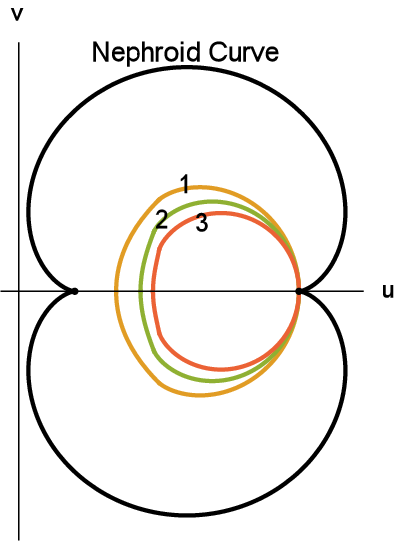}
\end{minipage}
\begin{minipage}{0.4\textwidth}
	  \small{
	  	1: Boundary of $q_\beta(\mathbb{D})$ with\\ 
	     $\beta=\frac{3\left(\sqrt{2}-\log \left(1+\sqrt{2}\right)+\log2\right)}{2}$.\\
	     \\
	2: Boundary of $\hat{q}_\beta(\mathbb{D})$ with\\
	 $\beta=\frac{\sqrt{2}+\log (2)-\log \left(1+\sqrt{2}\right)}{\log \left(\frac{5}{3}\right)}$.\\
	\\
	3: Boundary of $\tilde{q}_\beta(\mathbb{D})$ with\\
	 $\beta=\frac{5\left(\sqrt{2}-\log \left(1+\sqrt{2}\right)+\log2\right)}{2}$.
}
\end{minipage}
\captionof{figure}{The functions $q_\beta,\,\hat{q}_\beta$ and $\tilde{q}_\beta$ are defined in
(\ref{Proof-Thrm-LemB-Impl-Neph-Part-a-SRN}), (\ref{Proof-Thrm-LemB-Impl-Neph-Part-b-SRN}) and
(\ref{Proof-Thrm-LemB-Impl-Neph-Part-c-SRN}) of \Cref{Thrm-Crescent-Impl-Neph-SRN}.}
\label{Fig-Lune-Impl-Neph-SRN}
\noindent

\begin{theorem}
	Let $\varphi_{\scriptscriptstyle{C}}(z):=1+4z/3+2z^2/3$. Then, for $p\in\mathcal{H}$ with $p(0)=1$, each of the following subordinations is sufficient to imply that $p\prec\varphi_{\scriptscriptstyle{Ne}}$. Moreover, each estimate on $\beta$ is sharp.
	\begin{enumerate}[\rm(a)]
		\item $1+\beta{zp'}\prec\varphi_{\scriptscriptstyle{C}}(z)$ for $\beta\geq {5}/{2}$,
		\item $1+\beta{zp'}/{p}\prec\varphi_{\scriptscriptstyle{C}}(z)$ for $\beta\geq {5}/{3\log(5/3)}\approx3.26269$,
		\item $1+\beta{zp'}/{p^2}\prec\varphi_{\scriptscriptstyle{C}}(z)$ for $\beta\geq {25}/{6}$.
	\end{enumerate}
\end{theorem}
\begin{proof}
	\begin{enumerate}[(a):]
\item
		Consider the first order linear differential equation
		\begin{align}\label{Diff-Eqn-Cardioid-SRN}
	1+\beta{zq'_{\beta}(z)}=\varphi_{\scriptscriptstyle{C}}(z),
		\end{align}
	A simple calculation shows that the function $q_{\beta}(z)=1+{z(4+z)}/{3\beta}$ defined on $\overline{\mathbb{D}}$ is an analytic solution of \eqref{Diff-Eqn-Cardioid-SRN}. Proceeding as \Cref{Thrm-Crescent-Impl-Neph-SRN}\eqref{Proof-Thrm-Crescent-Impl-Neph-Part-a-SRN},
	 to get $\Theta(z)=zq'_{\beta}(z)\lambda(q_{\beta}(z))={4z}/{3}z+{2z^2}/{3}$ and $h(z)=\vartheta(q_\beta(z))+\Theta(z)=1+\Theta(z)$. Since $4z/3+2z^2/3$ is starlike, the function $\Theta$ is starlike, and hence $\mathrm{Re}\left(zh'/\Theta\right)=\mathrm{Re}\left(z\Theta'/\Theta\right)$ is positive in $\mathbb{D}$. Therefore, by \Cref{Lemma-3.4h-p132-Miller-Mocanu}, the first-order differential subordination $1+\beta{zp'}\prec\varphi_{\scriptscriptstyle{C}}=1+\beta{zq'_{\beta}}$ implies $p\prec{q_{\beta}}$. Now the claimed subordination  $p\prec\varphi_{\scriptscriptstyle{Ne}}$ holds if $q_{\beta}\prec\varphi_{\scriptscriptstyle{Ne}}$ holds true, which is possible if, and only if, $\varphi_{\scriptscriptstyle {Ne}}(-1)<q_\beta(-1)<q_\beta(1)<\varphi_{\scriptscriptstyle {Ne}}(1)$. Simplifying this condition, we see that $q_{\beta}\prec\varphi_{\scriptscriptstyle{Ne}}$ if $\beta\geq\max\left\{3/2,5/2\right\}=5/2$.
\item 
		Taking the function $\hat{q}_{\beta}:\overline{\mathbb{D}}\to\mathbb{C}$ as
		\begin{align*}
		\hat{q}_{\beta}(z)=\exp\left(\frac{z(4+z)}{3\beta}\right),
		\end{align*}
		and proceeding as in \Cref{Thrm-Crescent-Impl-Neph-SRN}(\ref{Proof-Thrm-Crescent-Impl-Neph-Part-b-SRN}) leads to the desired conclusion.
\item 
		Considering the function $\tilde{q}_\beta:\overline{\mathbb{D}}\to\mathbb{C}$ defined by
		\begin{align*}
		\tilde{q}_{\beta}(z)=\left(1-\frac{z(4+z)}{3\beta}\right)^{-1},
		\end{align*}
		and following the proof of  \Cref{Thrm-Crescent-Impl-Neph-SRN}(\ref{Proof-Thrm-Crescent-Impl-Neph-Part-c-SRN}) completes the proof.
		\qedhere
	\end{enumerate}
\end{proof}

\begin{theorem}\label{Thrm-Rational-Impl-Neph-SRN}
	Let $p\in\mathcal{H}$ with $p(0)=1$, and let
	\begin{align*}
	\varphi_{\scriptscriptstyle{0}}(z):=1+\frac{z}{k}\left(\frac{k+z}{k-z}\right),\quad k=\sqrt{2}+1.
	\end{align*}
	If any one of the following differential subordinations hold true, then $p\prec\varphi_{\scriptscriptstyle {Ne}}$. Each of the respective bounds on $\beta$ is best possible.
	\begin{enumerate}[\rm(a)]
		\item 
		$1+\beta{zp'}\prec\varphi_{\scriptscriptstyle{0}}$ for
		$\beta\geq 3\log\left(1+\frac{1}{\sqrt{2}}\right)-\frac{3}{2}\left(\sqrt{2}-1\right)\approx0.98308$,
		\item $1+\beta{zp'}/{p}\prec\varphi_{\scriptscriptstyle{0}}$ for
		$\beta\geq\frac{2 \left(\log \left(\frac{1+\sqrt{2}}{\sqrt{2}}\right)-\frac{1}{2 \left(1+\sqrt{2}\right)}\right)}{\log \left(\frac{5}{3}\right)}\approx1.28299$,
		\item $1+\beta{zp'}/{p^2}\prec\varphi_{\scriptscriptstyle{0}}$ for
		$\beta\geq5\left(\log\left(\frac{1+\sqrt{2}}{\sqrt{2}}\right)-\frac{1}{2 \left(1+\sqrt{2}\right)}\right)\approx1.63847$.
	\end{enumerate}
\end{theorem}
\begin{proof}
	\begin{enumerate}[(a):]
\item 
		Consider the analytic function 
		\begin{align*}
		q_{\beta}(z)=1+\frac{2}{\beta}\left(\log\left(\frac{k}{k-z}\right)-\frac{z}{2k}\right), \qquad z\in\overline{\mathbb{D}},
		\end{align*}
		satisfying the first-order linear differential equation $1+\beta{zq'_{\beta}}=\varphi_{\scriptscriptstyle{0}}$. Choosing $\vartheta(\xi)=1$ and $\lambda(\xi)=\beta$ in \Cref{Lemma-3.4h-p132-Miller-Mocanu}, we get $\Theta(z)=z(k+z)/k(k-z)$, which is starlike, and $h(z)=\varphi_{\scriptscriptstyle{0}}(z)$, which satisfies $\mathrm{Re}\left(zh'/\Theta\right)>0$ in $\mathbb{D}$. In light of \Cref{Lemma-3.4h-p132-Miller-Mocanu}, the differential subordination $1+\beta{zp'}\prec\varphi_{\scriptscriptstyle{0}}=1+\beta{zq'_{\beta}}$ implies $p\prec{q_{\beta}}$. Now the desired result follows if the subordination $q_{\beta}\prec\varphi_{\scriptscriptstyle {Ne}}$ holds, and, as in \Cref{Thrm-LemB-Impl-Neph-SRN}(\ref{Proof-Thrm-LemB-Impl-Neph-Part-a-SRN}), this is true if
		\begin{align*}
		\beta\geq\max\left\{
		\frac{3}{2}\left(1-\sqrt{2}+\log2\right),
		3\log\left(1+\frac{1}{\sqrt{2}}\right)-\frac{3}{2}\left(\sqrt{2}-1\right)
		\right\}.
		\end{align*}
\item 
		Taking the function $\hat{q}_{\beta}:\overline{\mathbb{D}}\to\mathbb{C}$ as
		\begin{align*}
		\hat{q}_{\beta}(z)=\exp\left(\frac{2\log\left(\frac{k}{k-z}\right)-\frac{z}{k}}{\beta}\right),
		\end{align*}
		and proceeding as in \Cref{Thrm-Crescent-Impl-Neph-SRN}(\ref{Proof-Thrm-Crescent-Impl-Neph-Part-b-SRN}) leads to the desired result $p\prec\varphi_{\scriptscriptstyle {Ne}}$.
\item 
		Noting that $\tilde{q}_\beta$ given by
		\begin{align*}
		\tilde{q}_{\beta}(z)=\left(1-\frac{2\log\left(\frac{k}{k-z}\right)-\frac{z}{k}}{\beta}\right)^{-1}, \qquad z\in\overline{\mathbb{D}}, 		\end{align*}
		is a solution of the differential equation $1+\beta{z\tilde{q}'_{\beta}}/\tilde{q}^2=\varphi_{\scriptscriptstyle{0}}$, the result can be easily established by following the steps of 
		 \Cref{Thrm-Crescent-Impl-Neph-SRN}(\ref{Proof-Thrm-Crescent-Impl-Neph-Part-c-SRN}).
\qedhere
	\end{enumerate}
\end{proof}

For $k=1+\sqrt{2}$ and $\theta\in\mathbb{R}$, we have
$$\left|\frac{k+e^{i\theta}}{k(k-e^{i\theta})}\right|\geq\frac{k-1}{k(k+1)}=3-2\sqrt{2}.$$
This shows that $|h(z)|\leq3-2\sqrt{2}$ is sufficient to conclude that $h\prec{z(k+z)/k(k-z)}$. Using this observation and the fact that for $f\in\mathcal{A}$, the function $p(z)=zf'(z)/f(z)\in\mathcal{H}$ satisfies $p(0)=1$, the following result easily follows from \Cref{Thrm-Rational-Impl-Neph-SRN}.
\begin{corollary}
	Let $f\in\mathcal{A}$ and $\mathcal{G}(z)$ be given by \eqref{Definition-J}. If any one of the following inequalities:
	\begin{enumerate}[\rm(a)]		
		\item
		 $\left|\left(\frac{zf'(z)}{f(z)}\right)\mathcal{G}(z)\right|\leq\frac{3-2\sqrt{2}}{3\log\left(\frac{1+\sqrt{2}}{\sqrt{2}}\right)-\frac{3}{2 \left(1+\sqrt{2}\right)}}\approx0.174526$
		\item
		$\left|\mathcal{G}(z)\right|\leq\frac{\left(3-2\sqrt{2}\right)\log\left(5/3\right)}{2\left(\log \left(\frac{1+\sqrt{2}}{\sqrt{2}}\right)-\frac{1}{2 \left(1+\sqrt{2}\right)}\right)}\approx0.133728$ 		
		\item $\left|\left(\frac{zf'(z)}{f(z)}\right)^{-1}\mathcal{G}(z)\right|\leq\frac{3-2\sqrt{2}}{5\left(\log\left(\frac{1+\sqrt{2}}{\sqrt{2}}\right)-\frac{1}{2 \left(1+\sqrt{2}\right)}\right)}\approx0.104716$.
	\end{enumerate}
	holds true, then $f\in\mathcal{S}^*_{Ne}$.
\end{corollary}

\begin{theorem}
	Let $\varphi_{\scriptscriptstyle{S}}(z):=1+\sin{z}$, and let $p\in\mathcal{H}$ satisfies $p(0)=1$. Each of the following subordinations imply $p\prec\varphi_{\scriptscriptstyle{Ne}}$:
	\begin{enumerate}[\rm(a)]
		\item 
		$1+\beta{zp'}\prec\varphi_{\scriptscriptstyle{S}}(z)$ for $\beta\geq\frac{3}{2}\sum_{n=0}^{\infty}\frac{(-1)^n}{(2n+1)!\times(2n+1)}\approx1.41912$,
		\item 
		$1+\beta{zp'}/{p}\prec\varphi_{\scriptscriptstyle{S}}(z)$ for $\beta\geq\frac{\sum_{n=0}^{\infty}\frac{(-1)^n}{(2n+1)!\times(2n+1)}}{\log(5/3)}\approx1.85207$,
		\item 
		$1+\beta{zp'}/{p^2}\prec\varphi_{\scriptscriptstyle{S}}(z)$ for $\beta\geq\frac{5}{2}\sum_{n=0}^{\infty}\frac{(-1)^n}{(2n+1)!\times(2n+1)}\approx2.36521$.
	\end{enumerate}
	The estimates on $\beta$ cannot be improved.
\end{theorem}
\begin{proof}
	\begin{enumerate}[(a):]
\item 
		Consider the first-order linear differential equation $1+\beta{zq'_{\beta}(z)}=\varphi_{\scriptscriptstyle{S}}(z)$. It is easy to verify that the analytic function
       $q_{\beta}(z)=1+\sum_{n=0}^{\infty}B_nz^{2n+1}$		
		is a solution of this differential equation, where
		$$B_n:=\frac{(-1)^n}{(2n+1)!\times(2n+1)}.$$ 	
		Following the proof of \Cref{Thrm-Crescent-Impl-Neph-SRN}(\ref{Proof-Thrm-Crescent-Impl-Neph-Part-a-SRN}), and noting that the function $\sin{z}$ is starlike in $\mathbb{D}$, we have the subordination implication:
		\begin{align*}
		1+\beta{zp'}\prec\varphi_{\scriptscriptstyle{S}}=1+\beta{zq'_{\beta}}\implies p\prec{q_{\beta}}.
		\end{align*}
		Again, as in \Cref{Thrm-Crescent-Impl-Neph-SRN}(\ref{Proof-Thrm-Crescent-Impl-Neph-Part-a-SRN}), the desired result $p\prec\varphi_{\scriptscriptstyle{Ne}}$ will follow if $\beta\geq {3}\sum_{n=0}^{\infty}B_n/{2}$.
\item 
        Verify that the function
        $\hat{q}_{\beta}(z)=\exp\left(\sum_{n=0}^{\infty}B_nz^{2n+1}/\beta\right)$
        satisfies the differential equation $1+\beta{z\hat{q}'_{\beta}(z)}/\hat{q}_\beta(z)=\varphi_{\scriptscriptstyle{S}}(z)$. Now follow the proof of \Cref{Thrm-Crescent-Impl-Neph-SRN}(\ref{Proof-Thrm-Crescent-Impl-Neph-Part-b-SRN}) for the rest. 
\item 
		Considering $\tilde{q}_\beta$ as
		\begin{align*}
		\tilde{q}_{\beta}(z)
		=\left(1-\frac{1}{\beta}\sum_{n=0}^{\infty}B_nz^{2n+1}\right)^{-1}, \qquad  z\in\overline{\mathbb{D}},
		\end{align*}
		and proceeding as in \Cref{Thrm-Crescent-Impl-Neph-SRN}(\ref{Proof-Thrm-Crescent-Impl-Neph-Part-c-SRN}) completes the proof.
\qedhere
	\end{enumerate}
\end{proof}

\section{Differential Subordinations Related to Janowski Class}
\label{Section-Janowski-Functions-SRN}
For $A,\,B\in[-1,1]$ with $B<A$, Janowski \cite{Janowski-1973-class-some-extremal} introduced the function class $\mathcal{P}[A,B]$ consisting of analytic functions of the form $h(z)=1+\sum_{n=1}^{\infty}c_nz^n$ satisfying $h(z)\prec\frac{1+Az}{1+Bz}$ for all $z\in\mathbb{D}$. Geometrically, $h\in\mathcal{P}[A,B]$ if, and only if, $h(0)=1$ and $h(\mathbb{D})$ is contained in the open disc having the line segment $\left[\frac{1-A}{1-B},\frac{1+A}{1+B}\right]$ as its diameter. Since $(1-A)/(1-B)\geq0$, $\mathrm{Re}\left(h(z)\right)>0$ for every $h\in\mathcal{P}[A,B]$. With certain conditions on $A$ and $B$, in this section, we find best possible lower bounds on the real $\beta$ such that the first-order differential subordination $1+\beta{zp'}/{p^j}\prec(1+Az)/(1+Bz)$ implies $p\prec\varphi_{\scriptscriptstyle{Ne}}$, where $j=0,1,2$.

\begin{theorem}\label{Thrm-Janowski-Impl-Neph-SRN}
	Let $-1<B<A\leq1,\,B\neq0$, and let $p\in\mathcal{H}$ satisfies $p(0)=1$. Then each of the following differential subordinations sufficiently ensures the subordination $p\prec\varphi_{\scriptscriptstyle{Ne}}$. Moreover, the respective estimates on the real $\beta$ are best possible. 
	\begin{enumerate}[\rm(a)]
		\item
		$1+\beta{zp'(z)}\prec\frac{1+Az}{1+Bz}$ for $\beta\geq\max\left\{\beta_1,\beta_2\right\}$, where
		\begin{align*}
		\beta_1=\frac{A-B}{2B}\log(1-B)^{-3} \text{ and } \beta_2=\frac{A-B}{2B}\log(1+B)^{3}.
		\end{align*}
		\item
		$1+\beta\left(\frac{zp'(z)}{p(z)}\right)\prec\frac{1+Az}{1+Bz}$ for $\beta\geq\max\left\{\beta_1,\beta_2\right\}$, where
		\begin{align*}
		\beta_1=\frac{A-B}{B\log3}\log(1-B)^{-1} \text{ and } \beta_2=\frac{A-B}{B\log(5/3)}\log(1+B).
		\end{align*}
		\item
		$1+\beta\left(\frac{zp'(z)}{p^2(z)}\right)\prec\frac{1+Az}{1+Bz}$ for $\beta\geq\max\left\{\beta_1,\beta_2\right\}$, where
		\begin{align*}
		\beta_1=\frac{A-B}{2B}\log(1-B)^{-1} \text{ and } \beta_2=\frac{A-B}{2B}\log(1+B)^{5}.
		\end{align*}
	\end{enumerate}
\end{theorem}
\begin{proof}
	\begin{enumerate}[(a):]
\item
	It can be easily verified that the analytic function $q_{\beta}(z)$ given by	
		\begin{align*}
		q_{\beta}(z)=1+\frac{A-B}{B\beta}\log(1+Bz), \quad  z\in\overline{\mathbb{D}},
		\end{align*}
		is a solution of $1+\beta{zq'_{\beta}(z)}=(1+Az)/(1+Bz)$. On defining
		  $\vartheta(\xi)=1$ and $\lambda(\xi)=\beta$, the functions $\Theta,h$ defined in \Cref{Lemma-3.4h-p132-Miller-Mocanu} take the form
		\begin{align*}
		\Theta(z)=zq'_{\beta}(z)\lambda(q_{\beta}(z))=\frac{(A-B)z}{1+Bz} \; \text{ and } \; h(z)=1+\Theta(z).
		\end{align*}
		Since $z\Theta'(z)/\Theta(z)=1/(1+Bz)$, so that for the given range of $B$, $\mathrm{Re}\left(z\Theta'(z)/\Theta(z)\right)>1/2>0$. This verifies that $\Theta$ is starlike in $\mathbb{D}$ and further establishes the positiveness of $\mathrm{Re}\left(zh'/\Theta\right)$ in $\mathbb{D}$. Thus, in light of \Cref{Lemma-3.4h-p132-Miller-Mocanu}, $1+\beta{zp'}\prec1+\beta{zq'_{\beta}}$ implies the subordination $p\prec{q_{\beta}}$. Now, the required result $p\prec\varphi_{\scriptscriptstyle{Ne}}$ holds if the subordination $q_{\beta}\prec\varphi_{\scriptscriptstyle{Ne}}$ holds. The subordination $q_\beta\prec\varphi_{\scriptscriptstyle{Ne}}$ holds if, and only if, $1/3<q_\beta(-1)<q_\beta(1)<5/3$. Calculation shows that this leads to the inequalities
		\begin{align*}
		\beta\geq3(B-A)\frac{\log(1-B)}{2B}=\beta_1 \text{ and } \beta\geq3(A-B){2B}\frac{\log(1+B)}{2B}=\beta_2.
		\end{align*}
		Therefore, the subordination $q_\beta\prec\varphi_{\scriptscriptstyle{Ne}}$ holds true if $\beta\geq\max\left\{\beta_1,\beta_2\right\}$.
\item
	Observe that the analytic function $\hat{q}_{\beta}:\overline{\mathbb{D}}\to\mathbb{C}$ given by
		\begin{align*}
		\hat{q}_{\beta}(z)=\exp\left(\frac{A-B}{B\beta}\log(1+Bz)\right),
		\end{align*}
		is a solution of $1+\beta{z\hat{q}'_\beta(z)/\hat{q}_\beta(z)}=(1+Az)/(1+Bz)$.
		Now proceeding as in \Cref{Thrm-LemB-Impl-Neph-SRN}(\ref{Proof-Thrm-LemB-Impl-Neph-Part-b-SRN}) leads to the desired subordination.
\item
		Considering the function $\tilde{q}_\beta:\overline{\mathbb{D}}\to\mathbb{C}$ given by
		\begin{align*}
		\tilde{q}_{\beta}(z)=\left(1-\frac{A-B}{B\beta}\log(1+Bz)\right)^{-1},
		\end{align*}
		and proceeding as in \Cref{Thrm-LemB-Impl-Neph-SRN}(\ref{Proof-Thrm-LemB-Impl-Neph-Part-c-SRN}) completes the proof.
		\qedhere
	\end{enumerate}
\end{proof}

\begin{remark}\label{Remark-Janowski-SRN}
Since
$$\left|\frac{(A-B)e^{i\theta}}{1+Be^{i\theta}}\right|\geq\frac{A-B}{1+|B|}, \qquad \theta\in\mathbb{R}$$
we conclude that
if $|g(z)|\leq(A-B)/(1+|B|)$ in $\mathbb{D}$, then $g\prec(A-B)z/(1+Bz)$.
\end{remark}
In view of the fact mentioned in \Cref{Remark-Janowski-SRN}, \Cref{Thrm-Janowski-Impl-Neph-SRN} yields the following sufficient conditions for  $\mathcal{S}^*_{Ne}$.
\begin{corollary}
	Let $f\in\mathcal{A}$ and $\mathcal{G}(z)$ be defined as in \eqref{Definition-J}. If any one of the following conditions hold true, then $f\in\mathcal{S}^*_{Ne}$.
	\begin{enumerate}[\rm(a)]
		\item
		 $\left|\left(\frac{zf'(z)}{f(z)}\right)\mathcal{G}(z)\right|\leq\frac{A-B}{1+|B|}\left(\max\left\{\beta_1,\beta_2\right\}\right)^{-1}$, where		
		\begin{align*}
		\beta_1=\frac{A-B}{2B}\log(1-B)^{-3} \text{ and } \beta_2=\frac{A-B}{2B}\log(1+B)^{3}.
		\end{align*}
		\item
		$\left|\mathcal{G}(z)\right|\leq\frac{A-B}{1+|B|}\left(\max\left\{\beta_1,\beta_2\right\}\right)^{-1}$, where
		\begin{align*}
		\beta_1=\frac{A-B}{B\log3}\log(1-B)^{-1} \text{ and } \beta_2=\frac{A-B}{B\log(5/3)}\log(1+B).
		\end{align*}
		\item $\left|\left(\frac{zf'(z)}{f(z)}\right)^{-1}\mathcal{G}(z)\right|\leq\frac{A-B}{1+|B|}\left(\max\left\{\beta_1,\beta_2\right\}\right)^{-1}$, where
		\begin{align*}
		\beta_1=\frac{A-B}{2B}\log(1-B)^{-1} \text{ and } \beta_2=\frac{A-B}{2B}\log(1+B)^{5}.
		\end{align*}
	\end{enumerate}
\end{corollary}

\section{Subordination Results Using Hypergeometric Functions}
\label{Section-GHGFs-SRN}
Let $a,b\in\mathbb{C}$ and $c\in\mathbb{C}\setminus\{0,-1,-2,\ldots\}$. Define
\begin{align}\label{Gaussian-HG}
F(a,b;c;z)={_2F_1}(a,b;c;z):=\sum_{j=0}^\infty\frac{(a)_j(b)_j}{j!\;(c)_j}\,z^j, \quad z\in\mathbb{D},
\end{align}
where $(x)_j$ is the Pochhammer symbol given by
\begin{align*}
(x)_j=
\begin{cases}
1, \quad j=0\\
x(x+1)(x+2)\cdots(x+j-1), \quad j\in\mathbb{N}:=\{1,2,\ldots\}.
\end{cases}
\end{align*}
The analytic function $F(a,b;c;z)$ given in \eqref{Gaussian-HG} is called the {\it Gaussian hypergeometric} function. The following properties of $F(a,b;c;z)$ will be used to prove our results. For further details, we refer to \cite{Rainville-Special-Functions-Book}.  
\begin{enumerate}[(i)]
	\item $F(a,b;c;z)$ is a solution of the differential equation
    \begin{align*}
    z(1-z)w''(z)+(c-(a+b+1)z)w'(z)-abw(z)=0.
    \end{align*}	
    \item $F(a,b;c;z)$ has a representation in terms of the gamma function
     \begin{align*}
     \Gamma(z)=\int_0^\infty t^{z-1}e^{-t}dt,\quad \mathrm{Re}(z)>0
     \end{align*} 
    as
    \begin{align}\label{Gamma-Function-Rep-Gaussian-HGF}
    F(a,b;c;z)=\frac{\Gamma(c)}{\Gamma(a)\Gamma(b)}\sum_{j=0}^\infty\frac{\Gamma(a+j)\Gamma(b+j)}{j!\;\Gamma(c+j)}\,z^j.
    \end{align}
    \item $F(a,b;c;z)$ satisfies
    \begin{align}\label{Derivative-Property-GHGF}
    F'(a,b;c;z)=\frac{ab}{c}F(a+1,b+1;c+1;z)
    \end{align}
    \item If $\mathrm{Re}\,c>\mathrm{Re}\,b>0$, then $F(a,b;c;z)$ has the following integral representation
    \begin{align}\label{Integral-Rep-Gaussian-HGF}
     F(a,b;c;z)=\frac{\Gamma(c)}{\Gamma(b)\Gamma(c-b)}\int_{0}^1\frac{t^{b-1}(1-t)^{c-b-1}}{(1-tz)^{a}}\,dt, \quad z\in\mathbb{D}.
    \end{align}
\end{enumerate}
We hereby mention that the function $zF(a,b;c;z)$ given by
\begin{align*}
zF(a,b;c;z)=z{_2F_1}(a,b;c;z)=z+\sum_{j=2}^\infty\frac{(a)_{j-1}(b)_{j-1}}{(j-1)!\;(c)_{j-1}}\,z^j, \quad z\in\mathbb{D},
\end{align*}
is known as {\it normalized} or {\it shifted} Gaussian hypergeometric function.
\subsection*{Order of Starlikeness}
Let $f\in\mathcal{A}$. The order of starlikeness (with respect to zero) of the function $f(z)$ is defined to be the number $\sigma(f)$ given by
\begin{align*}
\sigma(f):=
\inf_{z\in\mathbb{D}}\mathrm{Re}\left(\frac{zf'(z)}{f(z)}\right)\in[-\infty,1].
\end{align*}
In terms of $\sigma(f)$, we observe that $f\in\mathcal{A}$ is starlike if, and ony if, $\sigma(f)\geq0$, or precisely,
\begin{align*}
f\in\mathcal{S}^* \iff  \sigma(f) \geq 0.
\end{align*}
Related to the order of starlikeness of the modified Gaussian hypergeometric function $zF(a,b;c;z)$, K\"{u}stner
\cite{Kustner-2002-HGF-CMFT, Kustner-2007-HGF-JMAA} 
proved the following result.
\begin{lemma}[{K\"{u}stner \cite[Theorem 1 (a)]{Kustner-2007-HGF-JMAA}}]\label{Lemma-Kustner-2007-HGF-JMAA}
	If $0<a\leq b\leq c$, then 
	\begin{align*}
	1-\frac{ab}{b+c} 
	\leq \sigma\left(zF(a,b;c;z)\right)
	\leq 1-\frac{ab}{2c}.
	\end{align*}	
\end{lemma}
In this section, we use \Cref{Lemma-Kustner-2007-HGF-JMAA} along with \Cref{Lemma-3.4h-p132-Miller-Mocanu} to find sharp bounds on $\beta$ so that the first-order differential subordination
\begin{align*}
p(z)+\beta zp'(z)\prec\sqrt{1+z}, \text{ or }, 1+z
\end{align*}
implies the subordination $p\prec\varphi_{\scriptscriptstyle{Ne}}$.

\begin{theorem}\label{Thrm-LemB-Impl-Neph-GHGF}
	Let $p\in\mathcal{H}$ satisfies $p(0)=1$, and let
\begin{align*}
	p(z)+\beta zp'(z)\prec\varphi_{\scriptscriptstyle{L}}(z)=\sqrt{1+z},\quad \beta>0.
	\end{align*}
	If $\beta\geq\beta_L$, then $p\prec\varphi_{\scriptscriptstyle{Ne}}$, where $\beta_L$ is the unique root of
	\begin{align*}
	 \frac{3}{\Gamma(-\frac{1}{2})}\sum_{j=0}^\infty\frac{\Gamma(-\frac{1}{2}+j)}{j!\,(1+j\beta)}-1 = 0.
	\end{align*}
	The estimate on $\beta$ is best possible.
\end{theorem}

\begin{proof}
	An elementary analysis shows that the analytic function
	\begin{align}\label{Def-qbeta-Integral-Form-Lem}
q_\beta(z)=\frac{1}{\beta}\int_0^1\frac{t^{\frac{1}{\beta}-1}}{(1+zt)^{-1/2}}\,dt
	\end{align}
	is a solution of the linear differential equation $q_\beta(z)+\beta zq'_\beta(z)=\varphi_{\scriptscriptstyle{L}}(z)$. In view of the representation \eqref{Integral-Rep-Gaussian-HGF} of the Gaussian hypergeometric function, it is easy to see that the function $q_\beta(z)$ given by \eqref{Def-qbeta-Integral-Form-Lem} has the form
	\begin{align*}
	q_\beta(z)=F\left(-\frac{1}{2},\frac{1}{\beta};\frac{1}{\beta}+1;-z\right).
	\end{align*}
	For $\xi\in\mathbb{C}$, define $\vartheta(\xi)=\xi$ and $\lambda(\xi)=\beta$ so that
	\begin{align*}
	\Theta(z)=zq'_\beta(z)\lambda(q_\beta(z))=\beta zq'_\beta(z)=\beta zF'\left(-\frac{1}{2},\frac{1}{\beta};\frac{1}{\beta}+1;-z\right).
	\end{align*}
	This on using the identity \eqref{Derivative-Property-GHGF} gives
	  \begin{align}\label{Def-qbeta-GHGF-Form-Lem}
	  \Theta(z)=\frac{\beta}{2(1+\beta)} zF\left(\frac{1}{2},\frac{1}{\beta}+1;\frac{1}{\beta}+2;-z\right).
	  \end{align}
	  We now claim that the function $\Theta(z)$ given by \eqref{Def-qbeta-GHGF-Form-Lem} is starlike in $\mathbb{D}$ by showing that $\sigma(\Theta)\geq0$. 
For the normalized hypergeometric function on the right side of \eqref{Def-qbeta-GHGF-Form-Lem}, we have $a=1/2,b=1/\beta+1$ and $c=1/\beta+2$, so that the condition $0<a\leq b\leq c$ easily holds. Therefore, by \Cref{Lemma-Kustner-2007-HGF-JMAA}, it follows that
\begin{align*}
\sigma\left(zF(a,b;c;z)\right)
\geq 1-\frac{ab}{b+c}
=1-\frac{1+\beta}{2\left(2+3\beta\right)}
=\frac{3+5\beta}{2\left(2+3\beta\right)}
>0 
\qquad (\because\beta>0).	                                
\end{align*}
This shows that the hypergeometric function $zF\left(\frac{1}{2},\frac{1}{\beta}+1;\frac{1}{\beta}+2;-z\right)$ is starlike in $\mathbb{D}$ and the starlikeness of $\Theta(z)$ given by \eqref{Def-qbeta-GHGF-Form-Lem} follows.	  
  Furthermore, the function $h(z)=\vartheta\left(q_\beta(z)\right)+\Theta(z)=q_\beta(z)+\Theta(z)$ satisfies
  \begin{align*}
  \mathrm{Re}\left(\frac{zh'(z)}{\Theta(z)}\right)= \mathrm{Re}\left(\frac{1}{\beta}+\frac{z\Theta'(z)}{\Theta(z)}\right)>0,
  \end{align*}
  as $\Theta$ is starlike and $\beta>0$. In light of \Cref{Lemma-3.4h-p132-Miller-Mocanu}, we conclude that the subordination $p+\beta zp'\prec q_\beta+\beta zq'_\beta=\varphi_{\scriptscriptstyle{L}}$ implies $p\prec q_\beta$. The desired subordination $p\prec\varphi_{\scriptscriptstyle{Ne}}$ will now hold true if $q_\beta\prec\varphi_{\scriptscriptstyle{Ne}}$. As earlier, the subordination $q_\beta\prec\varphi_{\scriptscriptstyle{Ne}}$ holds if, and only if,
  \begin{align*}
  \varphi_{\scriptscriptstyle {Ne}}(-1)<q_\beta(-1)<q_\beta(1)<\varphi_{\scriptscriptstyle {Ne}}(1).
  \end{align*}
  On using the representation \eqref{Def-qbeta-GHGF-Form-Lem} and the identity \eqref{Gamma-Function-Rep-Gaussian-HGF}, the above condition yields
  \begin{align*}
  \frac{1}{3}\leq F\left(-\frac{1}{2},\frac{1}{\beta};\frac{1}{\beta}+1;1\right)=\frac{1}{\Gamma(-\frac{1}{2})}\sum_{j=0}^\infty\frac{\Gamma(-\frac{1}{2}+j)}{j!\,(1+j\beta)}
  \end{align*}
  and
  \begin{align*}
  \frac{5}{3} \geq F\left(-\frac{1}{2},\frac{1}{\beta};\frac{1}{\beta}+1;-1\right)=\frac{1}{\Gamma(-\frac{1}{2})}\sum_{j=0}^\infty\frac{\Gamma(-\frac{1}{2}+j)}{j!\,(1+j\beta)}(-1)^j.
  \end{align*}
  Or, equivalently,
  \begin{align*}
 \tau(\beta):= \frac{1}{\Gamma(-\frac{1}{2})}\sum_{j=0}^\infty\frac{\Gamma(-\frac{1}{2}+j)}{j!\,(1+j\beta)}-\frac{1}{3} \geq 0
  \end{align*}
  and
  \begin{align*}
  \delta(\beta):=\frac{5}{3} - \frac{1}{\Gamma(-\frac{1}{2})}\sum_{j=0}^\infty\frac{\Gamma(-\frac{1}{2}+j)}{j!\,(1+j\beta)}(-1)^j \geq 0.
  \end{align*}
We note that for $\beta\in(0,\infty)$, 
\begin{align*}
\delta(\beta)\in\left(\frac{5}{3}-\sqrt{2},\;\frac{2}{3}\right)
\quad \text{ and } \quad
\tau(\beta)\in\left(-\frac{1}{3},\frac{2}{3}\right).
\end{align*}
That is, as $\beta$ varies from $0$ to $\infty$, $\delta(\beta)$ is positive, while $\tau(\beta)$ takes positive as well as negative values. See the plots of $\tau(\beta)$ and $\delta(\beta)$ in \Cref{Plots-TauDel-HGF-SRN}. Further, $\tau(\beta)$ is strictly increasing in $(0,\infty)$. Therefore, both of the above required conditions hold true for $\beta\geq\beta_L$, where $\beta_L$ is the unique root of $\tau(\beta)$. This completes the proof.
\end{proof}
  \begin{figure}[H]
   \begin{subfigure}{0.45\textwidth}
   	\centering
     \includegraphics[scale=1.0]{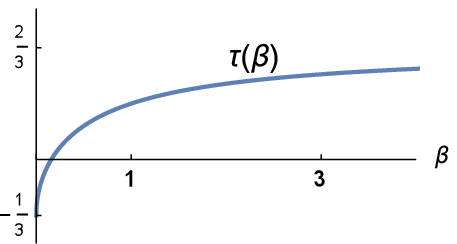}
  \end{subfigure}
  \begin{subfigure}{0.45\textwidth}
  	\centering
     \includegraphics[scale=1.0]{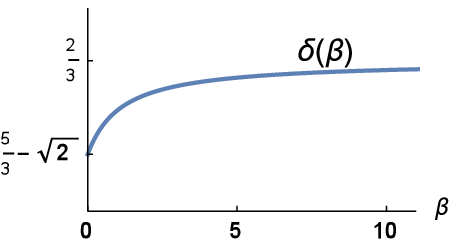}
   \end{subfigure}
   \caption{Plots of $\tau(\beta)$ and $\delta(\beta)$, $\beta>0$.}
   \label{Plots-TauDel-HGF-SRN}
  \end{figure}
The following result is a direct application of \Cref{Thrm-LemB-Impl-Neph-GHGF} obtained by setting $p(z)=zf'(z)/f(z)$.

\begin{corollary}
	Let $\mathcal{G}(z)$ be defined as in \eqref{Definition-J}, and let $f\in\mathcal{A}$ satisfies
	\begin{align*}
	\left(1+\beta\,\mathcal{G}(z)\right)\frac{zf'(z)}{f(z)}\prec \varphi_{\scriptscriptstyle{L}}(z).
	\end{align*}
	Then $f\in\mathcal{S}^*_{Ne}$ for $\beta\geq\beta_L$.
\end{corollary}

\begin{theorem}
	Let
	$p(z)+\beta zp'(z)\prec 1+z$, where $p\in\mathcal{H}$ satisfies $p(0)=1$ and $\beta>0$.
	Then $p\prec\varphi_{\scriptscriptstyle{Ne}}$ for $\beta\geq1/2$.
	The estimate on $\beta$ is sharp.
\end{theorem}

\begin{proof}
	Consider the differential equation $q_\beta(z)+\beta zq'_\beta(z)=1+z$ with the analytic function $q_\beta(z)$ given by
	\begin{align*}
	q_\beta(z)=\frac{1}{\beta}\int_0^1 t^{\frac{1}{\beta}-1}(1+zt)\,dt=F\left(-1,\frac{1}{\beta};\frac{1}{\beta}+1;-z\right),\quad z\in\mathbb{D}.
	\end{align*}
	as its solution. Defining the functions $\vartheta$ and $\lambda$ as in \Cref{Thrm-LemB-Impl-Neph-GHGF} we obtain 
	\begin{align*}
	\Theta(z)=zq'_\beta(z)\lambda(q_\beta(z))=\beta zq'_\beta(z)
	=\frac{\beta}{1+\beta} zF\left(0,\frac{1}{\beta}+1;\frac{1}{\beta}+2;-z\right)
	=\frac{\beta}{1+\beta} z,
	\end{align*}
	which is clearly a starlike function in $\mathbb{D}$. Also, $h(z)=\vartheta\left(q_\beta(z)\right)+\Theta(z)=q_\beta(z)+\Theta(z)$ satisfies
	$\mathrm{Re}\left({zh'}/{\Theta}\right)>0$ in $\mathbb{D}$. Therefore, it follows from \Cref{Lemma-3.4h-p132-Miller-Mocanu} that $p+\beta zp'\prec 1+z=q_\beta+\beta zq'_\beta$ implies the subordination $p\prec {q_\beta}$. To get the subordination $p\prec\varphi_{\scriptscriptstyle{Ne}}$, it now remains to prove that $q_\beta\prec\varphi_{\scriptscriptstyle{Ne}}$, which holds true if, and only if,
	$\varphi_{\scriptscriptstyle {Ne}}(-1)<q_\beta(-1)<q_\beta(1)<\varphi_{\scriptscriptstyle {Ne}}(1)$. This condition is equivalent to the conditions
	\begin{align*}
	 F\left(-1,\frac{1}{\beta};\frac{1}{\beta}+1;1\right)-\frac{1}{3}\geq0 \;
\text{ implying } \beta\geq \frac{1}{2}
	\end{align*}
	and
	\begin{align*}
	\frac{5}{3} - F\left(-1,\frac{1}{\beta};\frac{1}{\beta}+1;-1\right)\geq0 \; \text{ implying }	
\beta\geq -\frac{5}{2}
	\end{align*}
Thus $p\prec\varphi_{\scriptscriptstyle{Ne}}$ if $\beta\geq\max\{1/2,-5/2\}=1/2$. Since at $\beta=1/2$, $q_\beta(-1)=1/3$ as well as $q_\beta(1)=5/3$. This proves that the bound on $\beta$ can not be decreased further.
\end{proof}
\begin{remark}
	If the function $f\in\mathcal{A}$ satisfies the subordination
\begin{align*}
\left(1+\beta\,\mathcal{G}(z)\right)\frac{zf'(z)}{f(z)}\prec 1+z,
\end{align*}
then $f\in\mathcal{S}^*_{Ne}$ whenever $\beta\geq1/2$.
\end{remark}


\end{document}